\documentclass[a4paper]{amsart}
\usepackage{amssymb} 
\usepackage[pagewise]{lineno}
\usepackage{color}
\newtheorem{theo}{Theorem}[section]
\newtheorem{prop}{Proposition}[section]
\newtheorem{defi}{Definition}[section]
\newtheorem{lemma}{Lemma}[section]

\numberwithin{equation}{section} 

\definecolor{pavone}{rgb}{0.00,0.00,0.63}

\definecolor{malva}{rgb}{0.10,0.50,0.50}

\definecolor{rosso}{rgb}{1.,0.,0.}

\def\t{\textrm} 

\def\beq{\begin{eqnarray}}
\def\eeq{\end{eqnarray}}
\def\baa{\begin{array}}
\def\eaa{\end{array}}

\newcommand{\bdef}{\begin{definition}}
\newcommand{\be}{\begin{equation}}
\newcommand{\ee}{\end{equation}}
\newcommand{\bt}{\begin{theo}}
\newcommand{\et}{\end{theo}}
\newcommand{\bl}{\begin{lemma}}
\newcommand{\el}{\end{lemma}}
\newcommand{\bp}{\begin{prop}}
\newcommand{\epr}{\end{prop}}
\newcommand{\sgn}{\text{sgn}}

\def\uei{u^{\epsilon}_i}
\def\ue{u^{\epsilon}}
\def\en{\epsilon_n}
\def\ueio{u^{\epsilon}_{0i}}

\def\A{\mathcal A}
\def\C {{\mathbb C}}
\def\dsp{\displaystyle}
\def\ep{\epsilon}
\def\ov{\overline}

\let\dsp=\displaystyle
\def\R{{\mathbb R}}
\def\N{{\mathbb N}}

\def\li{\lambda_i}

\def\v{v_i}
\def\u{u_i}

\def\M{{\mathcal{M}}}
\def\Ii{I_i}

\def\iIi{\int_{\Ii}}

\def\suM{\sum_{i\in \M}}
\def\suMj{\sum_{j\in \M}}

\def\w{ \mathcal W^N}

\begin{document}
\title[VANISHING VISCOSITY APPROXIMATION ON STARSHAPED NETWORKS]
{VANISHING VISCOSITY APPROXIMATION FOR LINEAR TRANSPORT EQUATIONS ON FINITE STARSHAPED NETWORKS}
\author[Francesca R. Guarguaglini]{Francesca R. Guarguaglini$^*$}
%
\thanks{ \noindent
  $^*$ Dipartimento di Ingegneria e Scienze dell'Informazione e Matematica, \\
 \indent Universit\`a degli Studi di L'Aquila, Via Vetoio, I-67100 Coppito (L'Aquila), Italy
\\ \indent francescaromana.guarguaglini@univaq.it}
\author [Roberto Natalini]{ Roberto Natalini$^{**}$}
\thanks{ 
$^{**}$ Istituto per le Applicazioni del Calcolo "M.Picone", 
 Consiglio Nazionale delle Ricerche, Via dei Taurini 19, I-00185 Roma,
Italy\\
 \indent roberto.natalini@cnr.it}
%
\date{}
%
\begin{abstract} 
In this paper we study  linear parabolic equations 
 on a finite oriented starshaped network; the equations are    coupled by transmission conditions
set at the inner node, which do not impose continuity  on the unknown.
We consider this problem 
as a parabolic approximation of a set  of first order linear transport equations on the network 
and we prove that, when the diffusion coefficient vanishes, the family of solutions converges to the unique solution to the  first order equations and satisfies   suitable transmission conditions at the inner node, which are  determined by the
   parameters appearing  in the parabolic transmission conditions.
   \end{abstract}
\maketitle
\begin{section} 
{\bf Introduction}
In this paper we study a vanishing viscosity approximation for linear first order transport equations on a finite starshaped network
composed by $m$ arcs $I_i$, $i\in\M$, an inner node $N$ and $m$ outer nodes
$e_i$, $i\in\M$; each arc is a bounded interval $I_i=(0,L_i)$.
We set
$$\mathcal I=\{i\in\M: I_i \t{ is an incoming arc in the node } N\}\ ,$$
$$\mathcal O=\{i\in\M: I_i \t{ is an outgoing arc from the node } N\}\ ;$$
$\mathcal I$ and $\mathcal O$ are non empty sets.

On each arc $I_i$, $i\in\M$, we consider a   linear first order transport equation
\be\label {ipe}{\u}_t= -\li {\u}_x \qquad x\in I_i\  ,\ \ t>0\  \ ,\ee
where $\li>0$;
each equation (\ref{ipe}) is  complemented by the initial condition
\be
	\label{ici}
	u_i(x,0)=u_{0i}(x)\in  BV(I_i) \ ,\quad   
\ee
and,  when $i\in\mathcal I$, by the boundary condition  at the outer node $e_i$,
\be\label{bci2}\u(0,t)=B_i 
\ ,
\ \quad  t>0\  ,\quad \ i\in \mathcal I\ ;\ee
moreover   the equations (\ref{ipe}) for $i\in\M$, are coupled by the  following condition   at the inner node $N$,
\be\label{consflux}\dsp \sum_{i\in \mathcal I} \li u_i (L_i,t)=  \sum_{i\in \mathcal O} \li u_i (0, t) \ ,\ee
which establishes the conservation of the flux.

First order nonlinear conservation laws on networks with transmission conditions preserving the flux at the inner nodes
have been largely studied by several authors, see for instance  \cite{GP} and references therein for traffic models.

Here we focus on cases when equality (\ref{consflux}) is achieved  imposing at the node  this kind of conditions
\be\label{tci}
\li\u(0,t)= \sum_{j\in \mathcal I}  \gamma_{ij}\lambda_j u_j(L_j,t)\ , \ \quad  \forall i\in\mathcal O\  ,\ee
where 
\be\label{coga} 
 \gamma_{ij}\geq 0  \ \forall i\in\mathcal O\ \t { and } \forall j\in\mathcal I\ , \ \ \ \ \dsp \sum_{i\in\mathcal O} \gamma_{ij}=1\ \forall j\in\mathcal I\ .\ee

 We notice that  conditions (\ref{tci}) form  actually a set of  transmission conditions if the parameters $\gamma_{ij}$ satisfy
\be\label{coga2}  \t{ for all }i\in\mathcal O \t { there exists at least one } j\in\mathcal I \t { such that } \gamma_{ij}>0 \ .\ee

Set
$$L^{p}(\A)=\{ f=(f_1,f_2,...,f_m) : f _i\in L^p(I_i)\}\ ,
$$ 
$$W^{m,p}(\A)=\{ f=(f_1,f_2,...,f_m) : f_i\in W^{m,p}(I_i)\}\ ,$$
$$BV(\A)=\{ f=(f_1,f_2,...,f_m) : f_i\in BV(I_i)\}\ ,$$
$$BV(\A\times(0,T))=\{ f=(f_1,f_2,...,f_m) : f_i\in BV(I_i\times (0,T))\}	 ,$$
and $\Vert f\Vert=\dsp  \suM\Vert f_i\Vert$.

\begin{defi}\label{sol} For every $u_0\in BV(\A)$, a function $u\in BV(\A\times (0,T))$ is a solution of (\ref{ipe})-(\ref{bci2}), (\ref{tci})
 if,  for  all
$\phi_i\in  \mathcal C^1_0([0,L_i) \times  [0,T))$ ,
$$
	\int_0^T \iIi \left[\u(\phi_{i_t} +\li \phi_{i_x}) \right](x,t) dx dt + \iIi  {\u}_0(x) \phi_i(x,0) dx = -
	\li  B_i\int_0^T  \phi_i(0,t) dt\  ,\ \   i\in\mathcal I ,
$$
$$ \begin{array}{ll}
\dsp	\int_0^T \iIi  \left[\u(\phi_{i_t} +\li \phi_{i_x}) \right](x,t) dxdt &+\dsp  \iIi   {\u}_0(x) \phi_i(x,0) dx =
\\
 & -
\dsp	 \int_0^T  \phi_i(0,t)  \sum_{j\in \mathcal I}  \gamma_{ij}\lambda_j u_j(L_j,t) dt\  ,\quad i\in\mathcal O\  .
\end{array}$$
\end{defi}

For every $u_0\in BV(\A)$, problem 
(\ref{ipe})-(\ref{bci2}),(\ref{tci}) has a unique solution in the sense of the above definition; this fact is easily proved by
taking into accout that each component $u_i$ 
has null weak derivative in the direction
${\bf z}=(s,\lambda_i s)$ and then using  the Stampacchia's Lemma \cite{Sta} and  the Green's formula for $BV$ functions \cite{Vol}  (see also   \cite{NT}, Remark 3 and Theorem 3).

\medskip
In this paper we deal with a vanishing viscosity approximation  for  (\ref{ipe})-(\ref{bci2}),(\ref{tci}),
by considering the following parabolic problem, where $\ep$ is a variable parameter to be sent  to zero, 

\begin{equation}
	\label {para}\left\{ \begin{array}{ll}
	{\uei}_t= -\li {\uei}_x + \ep {\uei}_{xx} \qquad x\in I_i\ ,\ \ t>0\ ,\  i\in\M\ ,\\ \\
	\uei(x,0)=\ueio(x) \in W^{2,1}(I_i)\ , \qquad x \in I_i\ ,\quad i\in\M\ ,\\ \\
		\beta_{i}  \left( -\li \uei(N, t) +\ep {\uei}_x(N, t)\right)
	= \dsp \sum_{j\in\M} K_{ij} (\ue_j(N,t)-\uei(N,t))\ , \quad t>0\ , \ i\in\M, \\ \\
	\uei(e_i,t)=\ueio(e_i)\ ,\qquad  i\in\M
		\qquad\ \qquad\qquad t>0\ ,
	\end{array} \right.
\end{equation}
where $\u^\ep (N,t)$ denotes $\u^\ep (0,t)$ if $i\in\mathcal O$ and denotes  $\u^\ep (L_i,t)$ if $i\in\mathcal I$, and 
$\u^\ep (e_i,t)$, $ {u_0}_i^\ep (e_i)$ denote $u^\ep _i(0,t)$, $ {u_0}_i^\ep (0)$
 if $i\in\mathcal I$,  and denote $\u^\ep (L_i,t)$,  ${u_0}_i^\ep (L_i)$
 if $i\in\mathcal O$. Besides, we assume 
 \begin{equation} \label{17}
      	K_{ij}\geq 0\ ,\ K_{ij}=K_{ji}\  \forall i,j\in\M\ ,\ 
   	 \qquad  \beta_{i}=\left\{\begin{array}{ll} \ \  1\qquad i\in\mathcal I\\
   	-1\qquad i\in \mathcal O
   	 \end{array} \ .\right.  
 \end{equation}
   Finally, the values ${u_0}_i^\ep$  satisfy the transmission conditions in (\ref{para}) for all $i\in\mathcal \M$,
\be
	\label{uop} 
	\beta_{i}  \left( -\li \ueio(N) +\ep {\ueio}_x(N)\right)
	= \dsp \sum_{j\in\M} K_{ij} ({\ue_0}_j(N)-\ueio(N))\ 
\ee
and $u_0^\ep$ approximates $u_0$ for $\ep\to0$ in suitable way (see Section 3).

Transmission conditions in (\ref{para}) imply that the sum of the  fluxes  incoming in the inner node $N$ is equal to 
the one of the outgoing fluxes, thanks to the assumptions on the coefficients $K_{ij}$; however these conditions  impose no continuity condition on the solution at node, as assumed for instance in \cite{DM, CG, DZ, VZ}, 
so that they seem to be more appropriate when dealing with mathematical models for movements of individuals, as in traffic flows 
or in  biological phenomena involving bacteria or other microorganisms movements, where discontinuities  at the nodes  for the density functions 
are expected.
This kind of conditions were introduced by Kedem-Katchalsky in \cite{KK} as permeability conditions in the description of passive transport through a biological membrane, see also \cite{QVZ, CZ, CN} for various mathematical and related formulations; more recently they were proposed as transmission conditions for hyperbolic-parabolic and hyperbolic-elliptic 
systems describing movements of microorganism on  networks, driven by chemotaxis 
\cite{GN,G1,G2,GPS}.

In the next section we are going to prove that problem (\ref{para})-(\ref{uop}) has a unique smooth solution
$u^\ep$ in the following sense.

\begin{defi} For every $u^\ep_0\in W^{2,1}(\A)$ satisfying (\ref{uop}), a solution of (\ref{para}) is a function
$u^\ep   \in \mathcal C([0,+\infty);W^{2,1}(\A))\cap  \mathcal C^1([0,+\infty);L^1(\A))$ which solves the equation in (\ref{para}) in strong sense
and verifies the initial, boundary and transmission conditions in (\ref{para}).
\end{defi}

Moreover  we shall prove some  estimates  uniform in $\ep$ for $u^\ep$,
assuming the following condition on the coefficients $K_{ij}$ in (\ref{para}),
\be\label{anz}
\t{  for all } i\in\mathcal I\ ,  K_{ij}> 0 \t{ for some } j\in\mathcal O \t{ (at least one } j).
\ee

The uniform estimates  allow to use compacteness techniques  and prove   the convergence of each
sequence $\{u^{\ep_n}\}_{n\in\N}$, $\ep_n\to 0$, to the unique solution to problem (\ref{ipe})
which satisfies (\ref{ici}), (\ref{bci2}) and (\ref{tci}) with  coeffcients
$\gamma_{ij}$  determined by $K_{ij}$ and $\li$ ($i,j\in\M$). The result is achieved  provided  that
$\{u^{\ep_n}_0\}_{n\in \N}$ is a sequence approximating in suitable
way the initial data $u_0$ and assuming the boundary data $B_i$ on the incoming arcs; 
we show the actual possibility to obtain such approximation.

  The vanishing viscosity for  conservation laws on networks has  been approached in several papers. In particolar we refer to
  \cite{CG}, where the authors approximate a conservations law on a starshaped network with unlimited  arcs 
  by means of a parabolic problem, enforcing a continuity condition at the inner node beside to conservation of the parabolic flux; 
  also,  we refer to 
  \cite{CD} where a
  more general set of transmission conditions on the parabolic fluxes is considered, inspired  by \cite{GN}, which preserve the 
  parabolic flux but allows 
  discontinuities for the unknown; 
  both the papers consider a nonlinear conservation law 
  and the results 
  concern the convergence
   of a subsequence of  solutions of the parabolic problems 
   to a solution of the conservation law 
   satisfying the conservation of the flux at the inner node; the flux function $f$ in the conservation law 
  is such that
  $f(1)=f(0)=0$ and initial data verify $0\leq {u_0}_i\leq 1$, so that a uniform  $L^\infty$ bound for solutions to the   parabolic problems holds
  thanks to comparison results, and compensated compactness arguments are used to prove strong convergence. 
     
   Here, thanks to the linearity of the problems, not only  we can study in details the solvability of the regularized problem, 
  but also we are able to give a more precise result of convergence: once  the velocities $\li$ and the coeffcients $K_{ij}$
   in the conditions at the node $N$ in (\ref{para}) have been set, we identify univokely the limit function as the  solution to 
   (\ref{ipe})-(\ref{bci2}) satisfying (\ref{tci}) with specific $\gamma_{ij}$, 
   so that the convergence result holds for all the sequence
   $u^{\ep_n}$, $\ep_n\to 0$. The structure of the transmission conditions in (\ref{para}) is crucial in proving this result. 
   
   We remark that we could not  obtain a uniform  $L^\infty$ bound for the functions $\{u^\ep\}$ by means of  comparison techiques
 since,  in general, it is not easy to find supersolutions and subsolutions to the problem (\ref{para});  as a matter of fact,   the existence of solutions in the form  $U_i(x,t)=C_i$ satisfying  transmission conditions   (\ref{para})  is connected
   to the features of the coefficients matrix of the linear system 
   $$\dsp \sum_{j\in\M} K_{ij} (C_j-C_i)-\beta_i\li C_i =0\ , \ \ i\in\M;$$ moreover, 
   although the equations  in (\ref{para}) are linear, if $u$ is a solution , the function $u+C$ no longer satisfies
   the transmission conditions, so that  we are able to compare solutions only with the null one, i.e.
     the unique solution  which is constant on the whole network  and satisfies the  transmission conditions in (\ref{para}).
     
   One may wonder if, every given set of parameters $\{\gamma_{ij}\}$ satisfying (\ref{coga}) and (\ref{coga2}) 
   admits at least a corresponding set  $\{K_{ij}\}$ verifying (\ref{17}),(\ref{anz}) so that problems (\ref{para}) approximates 
   the problem (\ref{ipe})-(\ref{tci}).
   In this direction, first we prove that under the condition
   \be\label{lc}  \t{ for each } i\in\mathcal O\t{  there exists at least one } j\in\mathcal I \t{ such that } K_{ij } >0\ ,\ee
problem (\ref{ipe})-(\ref{tci}), whose solution is the  limit function  of the sequences $u^{\ep_n}$ ($\ep_n\to 0$),
  actually verifies condition (\ref{coga2}), i.e. (\ref{coga}) are actually transmission conditions.
  However, in the general case, we do not have the proof that
to each given family of parameters  $\gamma_{ij}$ satisfying (\ref{coga}), (\ref{coga2}),
corresponds at least one  family of 
coefficients $K_{ij}$ in (\ref{para}), satisfying (\ref{17}), (\ref{anz}), which allow to achieve the approximation result,
 since this involves too heavy computations. 
 In this paper, we are able to show how choosing the   coefficients 
  $K_{ij}$  in correspondence with  the parameters $\gamma_{ij}$
  only for some classes of networks and some classes of conditions (\ref{tci}).

  The paper is organized as follows. In Section 2 we study the solvability of problem (\ref{para})-(\ref{uop}) 
 by using the theory of semigroups. In Section 3 we derive a priori estimates for the solutions to 
(\ref{para})-(\ref{uop}), (\ref{anz}), which will turn out to be uniform in $\ep$ providing the family of initial data $u^\ep_0$ approximate 
$u_0$ in a suitable way. This approximation is described in Section 4, where, finally, we prove the convergence result
for  all the sequences of solutions of  the parabolic problem to a determined solution of (\ref{ipe})-(\ref{consflux}).
In the  last section we propose, for some classes of networks and some classes of transmission conditions (\ref{tci}),
sets of parameters $K_{ij}$ which give rise to limit function satisfying (\ref{tci}) with predetermined 
 coefficients $\gamma_{ij}$ .
\end{section}

\begin{section}
{\bf Existence of solutions for the parabolic problem}

This section is devoted to the prove of existence and uniqueness of solutions to problem (\ref{para})-(\ref{uop}).
It is convenient to rewrite the transmission condtions in (\ref{para}) in a more compact form, as follows
\be
	\label{cftc}
	\beta_{i}  \left(\li \uei(N, t) -\ep {\uei}_x(N, t)\right)
	= \dsp \sum_{j\in\M} \alpha_{ij} \ue_j(N,t)\ , \qquad t>0\ ,\ i\in\M\ ,
\ee
where 
\be
	\label{aij}
		\alpha_{ij}= -K_{ij}  \leq0\ \t{  for } i\neq j \  ,\ \qquad
	\alpha_{ii}=\sum_{j\in\M, j\neq i} K_{ij}\geq 0\ ;
\ee
notice that $\alpha_{ij}=\alpha_{ji}$,  and $\dsp \sum_{j\in\M} \alpha_{ij}=\dsp \sum_{i\in\M} \alpha_{ij}= 0$; moreover,
if (\ref{anz}) holds, then $\alpha_{ii}>0$ for $i\in\mathcal I$.

	Let us consider $q(x)=\{q_i(x)\}_{i\in\M}$, $q_i\in  \mathcal C^\infty(I_i)$, such that
$$  q_i(e_i)={u^\ep_0}_i (e_i)\ ,\quad q_i(N)=  q'_i(N)=0\ ,\ \ i\in\M\ ,$$
and the  operator $A_0^\ep$ defined as follows
\be
     \label{opar} \begin{array}{ll} D(A_0^\ep)=\left \{ \begin{array}{ll} w\in W^{2,1}(\A): 
      w(e_i)=0\ \  (i\in \mathcal \M) ,\\
      \beta_{i}\left(\li  w_i(N) -\ep {w_i}_x(N)\right)=\dsp  \sum_{j\in\M} \alpha_{ij} w_j(N) \ \   (i\in\M)  
      \end{array} \right\}
       \ ,\\
      A_0^\ep w= \{\ep {w_i}_{xx} -\li {w_i}_x\}_{i\in\M} \ .\end{array}
\ee

If $u^\ep$ is a solution to problem (\ref{para})-(\ref{uop}) then $v^\ep(t)=u^\ep(t)-q $  satisfies the problem
\be
	\label{PO} \left\{\begin{array}{ll}
	v^\ep(t)\in  \mathcal C([0,+\infty);  D(A_0^\ep))\cap  \mathcal C^1([0,+\infty);  L^1(\A))\ ,\\ 
	v^\ep_t= A_0^\ep v^\ep +F\ ,\\ 
	v^\ep (x,0)= u^\ep_0(x) - q (x)  \in  D(A_0^\ep)\ , \end{array}\right.
\ee
where $F(x)=\{F_i(x)= \ep q_i'' (x)-\li q'_i(x) \}_{i\in\M}$.

We  are going to prove the existence and uniqueness of solutions to problem (\ref{PO}) in order to prove the same results for problem 
(\ref{para})-(\ref{uop}).

The first step in the proof is showing  that, thanks to conditions (\ref{aij}) on $\alpha_{ij}$, the linear unbounded operator 
$A^\ep_0$  is m-dissipative in  $L^1(\A)$, so that the existence of a unique solution 
$v^\ep\in  \mathcal C([0,+\infty);D(A_0^\ep))\cap  \mathcal C^1([0,+\infty);L^1(\A))$ for   the homogeneous problem  associated to  (\ref{PO}) will result from the theory of linear semigroups (see \cite{CH}).
We recall that,  as defined in the same reference, a linear unbounded operator $A$  in a Banach space $X$ is dissipative  if 
$\Vert v-\theta Av \Vert_X\geq \Vert v\Vert_X$ 
 for all  $v\in D(A)$   and for all $\theta>0$; moreover
 a linear unbounded  operator $A$  in a Banach space $X$ is m-dissipative if it is dissipative and there exists $\theta_0>0$, such that for all $f\in X$, there exists a solution $v$ of $v-\theta_0 Av =f$.

In some of the proofs of the paper we are going to use the following functions.
Let $\delta>0$; 
we set 
\be
	\label{gg}
    g_\delta(w) = \left\{\begin{array}{ll} \qquad0 \qquad \qquad w\in(-\infty, 0]\\ 
    w^2 (4\delta )^{-1} \qquad \ \, w\in (0,2 \delta] \\
    w-\delta  \qquad\qquad w\in (2\delta, +\infty) \,
    \end{array}\right. \ .
\ee
Notice that 
 \be
 	\label{g}
	\begin{array}{ll}
    	\dsp g_\delta\in \mathcal C^1 (\R)\  , \ g_\delta(w)  \to_{\delta\to 0} [w]^+ \t{  pointwise }\  ,
    	\   \vert g_\delta(w)\vert\leq \vert w\vert\ , 
    	\\  
   	 \dsp g'_\delta  (w)\to_{\delta\to 0} {\chi}_{[0,+\infty)}(w) \t{ pointwise }\ ,\  \vert g'_\delta  (w)\vert \leq 1\ ,
     	\\  
     	\dsp g''_\delta(w)\to_{\delta\to 0} 0  \t{ pointwise } \ ,\  g''(w)w \t { is bounded } .
     	\end{array}
 \ee
Finally, we set 
\be
	\label{ovg}\ov g_\delta(w):=  g_\delta(w) + g_\delta(-w)\in  \mathcal C^1 (\R)\ .
\ee

\bp
	The operator $A_0^\ep$ defined in (\ref{opar}) is 
	dissipative in $L^1(\A)$.
\end{prop}
\begin{proof}
We have to prove that
$$
	  \suM \Vert \v -\theta (\ep {\v}_{xx} -\li{\v}_x)\Vert_1\geq \suM \Vert \v\Vert_1\
\ \qquad	 
 \forall v\in D(A_0^\ep)\ , \quad \forall \theta>0\ .
$$
Setting $f_i:= \v-\theta (\ep {\v}_{xx} -\li{\v}_x)$, then 
$$
	\suM \Vert f_i\Vert_1 \geq \suM\iIi f_i \ov g_\delta' (\v) dx
	=\suM\iIi \left( \v \ov g' (\v) +\theta (\li {\v}_x - \ep {\v}_{xx} ) \ov g_\delta' (\v)\right)dx 
$$
$$
	= \suM \iIi  \v \ov g' (\v) dx
	+ \theta \suM\iIi \left( \ep {\v}_{x} -\li {\v}) {\v}_x \ov g''_\delta (\v)\right)dx 
$$
$$
	+ \theta \sum_{i\in \mathcal I } \left( \li \v(L_i) -\ep{\v}_x(L_i) \right)\ov g'_\delta(\v(L_i))
	-  \theta \sum_{i\in \mathcal O} \left( \li \v(0) -\ep{\v}_x(0)\right) \ov g_\delta'(\v(0))
$$
$$
	\geq \suM \iIi  \v \ov g_\delta' (\v) dx +\theta \suM\iIi \left( -\li {\v}) {\v}_x \ov g''_\delta (\v)\right)dx 
$$
$$
 	+ \theta \sum_{i\in \mathcal I} \left(\sum_{j\in\M} \alpha_{ij} v_j(N)\right) \ov g_\delta' (\v(N))
	+ \theta \sum_{i\in \mathcal O} \left(\sum_{j\in\M} \alpha_{ij} v_j(N)\right) \ov g_\delta' (\v(N))\ 
$$
$$
	=\suM \iIi  \v \ov g_\delta' (\v) dx
	+ \theta \suM\iIi \left( -\li {\v}) {v_i}_x \ov g''_\delta (\v)\right)dx 
  	+\theta  \sum_{i,j \in \M}  \alpha_{ij} v_j(N) \ov g_\delta'(\v(N))\ .
$$
Since   $v_i, {v_i}_x\in L^1(I_i)$, 
thanks to (\ref{g}), when $\delta$ goes to zero, 
using the dominated convergence theorem, we obtain 
$$
	\suM \Vert f_i\Vert_1 \geq
	 \suM \Vert \v\Vert_1+ \theta\sum_{i,j \in \M}  \alpha_{ij} |v_j(N)| 
	 \sgn(v_j(N)) \sgn(\v(N) )
$$
 $$
 	= \suM \Vert \v\Vert_1+\theta  \sum_{j \in \M}  |v_j(N)| 
	 \sum_{i\in \M}
  	\alpha_{ij}  \sgn(v_j(N))
  	\sgn(\v(N)) \ .
$$
  The conditions  (\ref{aij}), which are valid  thanks to the assumptions on the coefficients      $K_{ij}$,  ensure that there holds
  $$
  	 \sum_{i\in \M}
  	\alpha_{ij}  \sgn(v_j(N))
  	\sgn(\v(N)) \geq 0 \ ,
$$
  so the dissipativity of $A_0^\ep$ follows.
  \end{proof}
  
  \bp 
  The operator $A_0^\ep$ defined in (\ref{opar}) is 
m-dissipative in $L^1(\A)$.
\end{prop}
\begin{proof}
 Since the operator $A_0^\ep$ is  dissipative, we have only 
 to prove that  for all $f\in L^1(\A)$ there exists a solution $v\in D(A_0^\ep)$ of the equation
  $v-A_0^\ep v=f$. Let $f=\{f_i\}_{i\in\M}$; if $f_i\in  \mathcal C^\infty_0(\Ii)$ for all $i\in\M$, the solutions 
  of the equation
  \be
  	\label{eqlin2}{\v''} -\frac \li \ep {\v'} -\frac 1 \ep \v =\frac{ f_i}{\ep} \ ,\qquad  x\in\Ii
\ee
  can be written
\be\label{intgen}
  	\dsp \v (x) = c_{i1} e^{a_{i1} x}+ c_{i2} e^{a_{i2} x} +p_i(x)
\ee
  where 
  \begin{equation}
  	   \label{radici} \ep\,{a_{i1}} =\frac \li 2 -\frac 1 2 \sqrt{\li^2 +4\ep}\  < 0
 	    \ ,\qquad
  	   \ep\, {a_{i2}} =\frac \li 2 +\frac 1 2 \sqrt{\li^2 +4\ep} \ > \li \ ,
  \end{equation}
  and $p_i(x)$ is the particular solution to (\ref{eqlin2}) 
     which satisfies the conditions
  \be   \label{intpart}
  	p_i(L)=p'_i(L)=0\ \t{ if } \  i\in \mathcal I\ ,\qquad  p_i(0)=p'_i(0)=0 \ \it{ if }\   i\in \mathcal O \ .
  	  \end{equation}
  
  We are going to prove that there exist $c_{i1} , c_{i2} \in\R$ $(i\in\M)$ such that $v=\{\v\}_{i\in\M}$ given by (\ref{intgen}) satisfies the boundary and the transmission conditions in (\ref{opar}) .
  The boundary  conditions read 
  \be\label{bcci} 
  	\begin{array}{ll}
	c_{i1}+c_{i2} +p_i(0) =0 \qquad i\in\mathcal I\ , \\ \\ 
  	c_{i1} e^{a_{i1} L_i}+c_{i2}  e^{a_{i2} L_i}+p_i(L_i) =0 \qquad i\in\mathcal O\ ,
	\end{array}
\ee
while the conditions at the inner node read
\be\label{s1}
  	\left\{  \begin{array}{ll} \li \left(c_{i1} e^{a_{i1} L_i}+ c_{i2} e^{a_{i2} L_i} \right) 
  	-\ep \left(c_{i1} a_{i1}e^{a_{i1} L_i}+ c_{i2} a_{i2} e^{a_{i2} L_i}\right) \\ \\
  	 =\dsp \sum_{j\in\mathcal I} \alpha_{ij} \left( c_{j1} e^{a_{j1} L_j}+ c_{j2} e^{a_{j2} L_j} \right) 
  	+ \dsp \sum_{j\in\mathcal O} \alpha_{ij} \left( c_{j1} + c_{j2}  \right)\ ,\qquad \qquad  i\in \mathcal I,
  	\\ \\
	 - \li \left(c_{i1} + c_{i2}\right) 
  	+\ep \left(c_{i1} a_{i1}+ c_{i2} a_{i2} \right) 
  	=\dsp  \sum_{j\in\mathcal I} \alpha_{ij} \left( c_{j1} e^{a_{j1} L_j}+ c_{j2} e^{a_{j2} L_j} \right) \\ 
  	 +\dsp \sum_{j\in\mathcal O} \alpha_{ij} \left( c_{j1} + c_{j2}\right) \ \qquad \qquad\qquad\qquad\qquad\qquad\qquad \qquad\qquad i\in \mathcal O \ .
   	\end{array}  \right.
\ee

Using the conditions (\ref{bcci}) in (\ref{s1}) and setting 
$$
	\mathcal P_i:=  e^{a_{i1} L_i} (\li  - \ep a_{i1}  ) \ p_i(0)
	- \sum_{j\in\mathcal I} \alpha_{ij}  e^{a_{j1} L_j} p_j(0) 
	 -\sum_{j\in\mathcal O} \alpha_{ij}  e^{-a_{j1} L_j} p_j(L_j)  ,\quad i\in\mathcal I,
$$
$$
	\mathcal P_i:=e^{-a_{i1} L_i}  (-\li  + \ep a_{i1} )p_i(L_i)
	- \sum_{j\in\mathcal I} \alpha_{ij}  e^{a_{j1} L_j} p_j(0) 
	 -\sum_{j\in\mathcal O} \alpha_{ij}  e^{-a_{j1} L_j} p_j(L_j), \quad  i\in\mathcal O,
$$
  $\{ c_{i2}\}_{i\in\M}$ will result as solution of the linear systems of $m$ equations:
$$
	\left\{\begin{array}{ll}
	 	 \dsp  \left ( (\ep a_{i2}-\li) e^{a_{i2} L_i}+ (\li-\ep a_{i1}) e^{a_{i1} L_i} +
	 	 \alpha_{ii} (e^{a_{i2} L_i}-e^{a_{i1} L_i})\right)  c_{i2}
	\\ \\
		\dsp +\sum_{j\in \mathcal I, j\neq i}  \alpha_{ij}  \left( e^{a_{j2} L_j}  -e^{a_{j1} L_j}\right) c_{j2}+
		\sum_{j\in \mathcal O } \alpha_{ij}  \left(1-  e^{(a_{j2}- a_{j1}) L_j}  \right) c_{j2}=-\mathcal P_i\ ,
		\ \ \ \  i\in\mathcal I,
	\\ \\
		\dsp  \left ( (\ep a_{i2}-\li) + (\li-\ep a_{i1}) e^{(a_{i2} -a_{i1}) L_i} +
		\alpha_{ii} (e^{(a_{i2} -a_{i1}) L_i} -1) \right) c_{i2}
	\\ \\
		\dsp -\sum_{j\in \mathcal I}  \alpha_{ij}  \left( e^{a_{j2} L_j}  -e^{a_{j1} L_j}\right) c_{j2}+
		\sum_{j\in \mathcal O, j\neq i } \alpha_{ij}  \left(  e^{(a_{j2}- a_{j1}) L_j} -1 \right) c_{j2}=\mathcal P_i\ ,
		\ \  \ \ \   i\in\mathcal O .
	\end{array}\right.
$$
 \smallskip

Let $H=\{h_{ij}\}$  be the coefficients matrix of the above linear system; then (\ref{radici})  and the assumptions on $\alpha_{ij}$ 
imply that  
$$
	\begin{array}{ll}
		\dsp \sum_{j\in\mathcal M, j\neq i} \vert h_{ji} \vert = (e^{a_{i2} L_i}-e^{a_{i1} L_i}) \sum_{j\in\mathcal M, j\neq i} |\alpha_{ji}|
		\leq \alpha_{ii} (e^{a_{i2} L_i}-e^{a_{i1} L_i}) 
	\\ \\
		\dsp \qquad <  (\ep a_{i2}-\li) e^{a_{i2} L_i}+ (\li-\ep a_{i1}) e^{a_{i1} L_i} +
		\alpha_{ii} (e^{a_{i2} L_i}-e^{a_{i1} L_i})=h_{ii}\ ,
	\end{array}\ \ \ \ 
	i\in\mathcal I\ ,
$$
$$
	\begin{array}{ll}
		\dsp \sum_{j\in\mathcal M, j\neq i} \vert h_{ji} \vert = ( e^{(a_{i2} -a_{i1}) L_i}-1) \sum_{j\in\mathcal M, j\neq i} |\alpha_{ji}|
		\leq \alpha_{ii}  ( e^{(a_{i2} -a_{i1}) L_i}-1)
	\\ \\
		\dsp \qquad <  (\ep a_{i2}-\li) + (\li-\ep a_{i1}) e^{(a_{i2} -a_{i1}) L_i} +
		\alpha_{ii} (e^{(a_{i2} -a_{i1}) L_i} -1)
		=h_{ii}\ ,
	\end{array}\ \ 
	i\in\mathcal O\ ,
$$
so that $H$ has strictly dominant diagonal . It follows that there exists a unique  solution $\{ c_{i2}\}_{i\in\M}$ of the linear system,
and
$\{ c_{i1}\}_{i\in\M}$ con be determined using (\ref{bcci}).

If $f\in L^1(\A)$ then we consider a sequence $\{f_n\}_{n\in\N}$, ${f_n}_i\in C_0^\infty(I_i)$ ( $i\in\M$),  such that ${f_n}\to_{n\to +\infty} f$ in $L^1(\A)$; moreover we consider  the corresponding sequences $\{p_n\}_{n\in\N}$, satisfying (\ref{intpart}), 
and the functions
$$
  	\dsp {v_n}_i (x) = c^n_{i1} e^{a_{i1} x}+ c^n_{i2} e^{a_{i2} x} +{p_n}_i(x)\ ,\qquad i\in\M\ ,
$$
solutions to (\ref{eqlin2}), where $f$ is replaced by $f_n$,
\be
  	\label{eqlinn}{{v''_n}_i} -\frac \li \ep {{v'_n}_i} -\frac 1 \ep {v_n}_i =\frac{{f_n}_i }{\ep}\ ,\qquad  x\in\Ii\ ,\ i\in\M\ .
\ee

The sequence $\{v_n\}_{n\in\N}$ converges to some $\ov v$ in $L^1(\A)$, since the operator $A_0^\ep$ is dissipative in $L^1(\A)$  so that
$\Vert v_n-v_{\ov n}\Vert_1 \leq \Vert f_n-f_{\ov n}\Vert_1$, for $n,\ov n\in\N$.

We can express  each ${p_n}_i$,  for $i\in \mathcal I$, as follows 
$$
	{p_n}_i(x) = \frac 1 {a_{2i}-a_{1i}} 
	 \int_{L_i}^x \frac{{f_n}_i(y)}{\ep} \left( -e^{a_{i1}(x-y)}  +
	e^{a_{i2}(x-y)}  \right) dy  \ ,
$$
 and a similar expression holds for $i\in \mathcal O$; if we define, for $i\in \mathcal I$
$$
	{\ov p}_i(x) =  \frac 1 {a_{2i}-a_{1i}} 
	 \int_{L_i}^x \frac{{f}_i(y)}{\ep} \left( -e^{a_{i1}(x-y)}  +
	e^{a_{i2}(x-y)}  \right) dy  \ ,
$$
we obtain 
$$
	\vert {p_n}_i(x) -{\ov p}_i(x)\vert 
	\leq \frac 1 {a_{2i}-a_{1i}} 
	\int_0^{L_i} \frac{|{f_n}_i(y) - {f}_i(y)|}{\ep}\left(e^{a_{i2}(x-y)} 
	-e^{a_{i1}(x-y)} 
	\right) dy \ ,
	$$
so that $\{{p_n}_i\}$ converges in $\mathcal C(\ov I_i)$ to $\ov p_i$ for each  $i\in \mathcal I$, and similarly, for each
for $i\in \mathcal O$. This fact implies   the convergence of the sequences $\{  c^n_{i2}\}_{n\in\N}$ 
and $\{ c^n_{i1}\}_{n\in\N}$, 
for all $i\in\M$.
Then 
$$
	\ov v_i(x) = \ov c_{i1} e^{a_{i1} x}+ \ov c_{i2} e^{a_{i2} x} +{\ov p}_i(x)
$$ 
and the claim is proved  
 if we show that $\ov v\in D(A_0^\ep)$ and it satisfies (\ref{eqlin2}).

Since, for all $i\in\M$, 
$$
	{p'_n}_i(x) = \frac 1 {a_{2i}-a_{1i}}
		 \int_{L_i}^x {f_n}_i(y) \left(-a_{1i} e^{a_{1i}(x-y)} 
	+ a_{2i} e^{a_{2i}(x-y)} 
		\right) dy \ ,
$$
the functions $\ov p$ belong to $W^{1,1}(\A)$ and 
$$
	\ov p'_i(x) = \frac 1 {a_{2i}-a_{1i}}
	 \int_{L_i}^x f_i(y) \left(-a_{1i} e^{a_{1i}(x-y)} 
	+ a_{2i} e^{a_{2i}(x-y)} 
		\right) dy \ ;
$$
moreover 
$$
	{v'_n}_i(x)= a_{1i} c^n_{1i} e^{a_{1i}x} +a_{1i} c^n_{2i} e^{a_{2i}x} + {p'_n}_i(x)\ ,
$$
$$
	{\ov v'}_i(x)= a_{1i} \ov c_{1i} e^{a_{1i}x} +a_{1i} \ov c_{2i} e^{a_{2i}x} + {\ov p'}_i(x)\ ,
$$
so we infer that 
$$
	\dsp { v_n}_i  \rightarrow_{n\to +\infty}{\ov v}_i  \ \ \t{  in } W^{1,1}(I_i)\ .
$$
From the equation (\ref{eqlinn}) we obtain 
$$-\int_0^{L_i} {\ov v'_i}\phi' = \int_0^{L_i} \phi\left(  \frac \li \ep  {\ov v'_i} +\frac 1 \ep \ov v_i +\frac{f_i}{\ep}\right)\ ,\ \ 
\phi\in C^1_0(I_i)\ ,$$
so $\ov v_i\in W^{2,1}(I_i)$,
$
       |{ v_n}_i -{\ov v}_i | \to 0  \t{ in } W^{2,1}(I_i)
       $
and 
$\ov v_i$ 
satisfies (\ref{eqlin2}) for all $i\in\M $.

Finally, the convergence in $W^{2,1}(I_i)$ implies that 
$\ov v$ satisfies the boundary and transmission conditions . 
\end{proof}

The proof of existence and uniqueness  of solutions to problem  (\ref{PO}), where homogeneous boundary conditions are considered, 
is a direct consequence of the previous propositions.

\bp \label{esistenzaparaomo}
For  every $u^\ep_0\in D(A_0^\ep)$, there exists a unique solution to 
problem (\ref{PO}).
\end{prop}
\begin{proof}
For all $\ep>0$ we easily see that $D(A_0^\ep)$ is dense in  $L^1(\A)$, in fact
$$\{f=(f_1,f_2,...,f_m): f_i\in  \mathcal C^{\infty}_0(I_i)\} \subset D(A_0^\ep)\ .$$
Since $A^\ep$ is an m-dissipative operator in $L^1(\A)$ and $D(A_0^\ep)$ is dense in  $L^1(\A)$,  the Hille -Yosida theorem  implies  that  $A_0^\ep$  generates a linear  semigroup $\mathcal \tau_\ep$. 
Moreover,  $F_i\in  \mathcal C^\infty (I_i)$,  so  the  linear contraction semigroups theory applies  to     problem (\ref{opar}),  (\ref{PO})  (see \cite{CH}), providing the  unique solution $v^\ep \in  \mathcal C([0,+\infty);D(A^\ep_0))\cap  \mathcal C^1([0,+\infty);L^1(\A))$,
$$v^\ep (t) = \mathcal \tau_\ep (t) (u^\ep_0-q)+\int_0^t \mathcal \tau_\ep (t-s) F\ ds \ .$$ 
\end{proof}
\end{section}

The existence and uniqueness of solutions to problem (\ref{para}))-(\ref{uop}) immediately follows from the above proposition, thanks to the 
position $v^\ep=u^\ep -q$ set at the beginning of this section.

\bt \label{esistenzapara}
For every $u^\ep_0\in W^{2,1}(\A)$ satisfying (\ref{uop}) there exists a unique solution $u^\ep$ to 
problem  (\ref{para}) with
$$
	u^\ep\in  \mathcal C([0,+\infty);W^{2,1}(\A))\cap  \mathcal C^1([0,+\infty);L^1(\A))\ ,
$$
and this solution is given by
\be
\label {espu}
	u^\ep (t) = \mathcal \tau_\ep (t) (u^\ep_0-q)   +\int_0^t \mathcal \tau_\ep(t-s) F ds  + q \ .
\ee

\end{theo}

\begin{section}
{\bf Uniform estimates for the solutions of the parabolic problem}

In this section we obtain  some estimates  for the solutions to problems (\ref {para})-(\ref{uop}),
which will be uniform in $\ep$ provided that $\ueio$
 approximates $u_{i0}$ in suitable way when $\ep\to 0^+$ (see next Section).

Let $0<\ep\leq 1$ and let $q$ be the function defined at the beginning of  Section 2.
\begin{prop}\label{stima1}

Let $u^\ep$ be the solution to problem (\ref{para})-(\ref{uop});  then for all $t>0$,
 $$
 	\Vert u^\ep(t)\Vert_{L^1(\A)} \leq \Vert u^\ep_0\Vert_{L^1(\A)} + \eta_1 t+\eta_2\ ,
$$
 where $\eta_1,\eta_2\geq 0$.
\end{prop}
\begin{proof}
From Theorem \ref{esistenzapara} we know that
$$
	u^\ep (t) = \mathcal \tau_\ep (t) (u^\ep_0-q)   +\int_0^t \mathcal \tau_\ep(t-s) F ds  + q
$$
so, thanks to the properties of contraction semigroups,
$$
	 \Vert \ue \Vert_{L^1(\A)} \leq  \Vert \ue_0 \Vert_{L^1(\A)} + 2 \Vert q\Vert_{L^1(\A)} +t  \suM \Vert \ep q_i'' (x)-\li q'_i(x) \Vert_{L^1(\A)}
	  \  .
$$
\end{proof}

Using the properties  of the semigroup $\mathcal \tau_\ep$ we can also prove the following estimate for $u^\ep_t$.

\begin {prop}\label {stima3}
Let $u^\ep$ be the solution to problem (\ref{para})-(\ref{uop}), 
 then, for all $t>0$,
$$
	 \Vert \ue_t (t)\Vert_{L^1(\A)} \leq
	\suM  \Vert -\li {{u^\ep_0}}_x +\ep   {{u^\ep_0}}_{xx}\Vert_{L^1(\A)} 
	 	\ .
$$
\end{prop}
\begin{proof}
Since $(\ue_0 -q) \in D(A^\ep_0)$ the following equality holds
$$A^\ep_0 \mathcal \tau_\ep (u_0^\ep -q) =  \mathcal \tau_\ep A^\ep_0 (u_0^\ep -q)\ ,$$
hence, using (\ref{espu}) (see \cite{CH}), we can write

$$
	 \Vert u^\ep_t(t) \Vert_{L^1(\A)}= \Vert \left( \mathcal \tau_\ep(t)(u^\ep_0-q)\right)_t + \mathcal \tau_\ep(t) F \Vert_{L^1(\A)}
$$
$$
	=  \Vert A^\ep_0 \mathcal \tau_\ep(t)(u^\ep_0-q) +  \mathcal \tau_\ep (t)F \Vert_{L^1(\A)} =  
	 \Vert  \mathcal \tau_\ep(t)(A^\ep_0(u^\ep_0-q)+ F) \Vert_{L^1(\A)}
$$
$$
	\leq \Vert  A^\ep_0(u^\ep_0 - q) + F \Vert_{L^1(\A)}
	= \Vert {\ep u^\ep_0}_{xx} -\li {u^\ep_0}_x \Vert_{L^1(\A)}  \ .
$$ 
\end{proof}

Now  we can prove the following estimate for  $u^\ep_x$.
\begin{prop} \label{ux2} 
Let (\ref{anz}) hold
and let $\ue$ be the solution to problem (\ref{para})-(\ref{uop}); then, for all $T>0$  and all  $i\in\M$, for small $\ep$,

$$
	\begin{array}{ll}
	\dsp   \sup_{t\in[0,T]} \int_{I_i}  \vert {\u^\ep}_{x} (x,t)\vert dx 
	&
	\leq  \dsp C\sum_{j\in\mathcal M} \Vert   {u^\ep_{0j}}\Vert_{W^{1,1}(I_j)} 
	\\ \\
	\dsp 
	&+ C
	 \dsp \sup_{t\in[0,T]}  \sum_{j\in\mathcal M}\left(
     \Vert  {u_j^\ep }_{t} ( t)\Vert_{L^1(I_j)}+\Vert u_j^\ep ( t)\Vert_{L^1(I_j)} \right) \ ,
   \end{array} 
	$$
where $C$ depends on $\lambda_\iota$,$L_\iota$, $ K_{\iota j}$ ($\iota,j\in\M$) and Sobolev constants.
\end{prop}
\begin{proof}
Using  the function (\ref{ovg}) we can write,
for each  $i\in\M$ and $t>0$, 
 $$
 	\int_{I_i} \left[{\u^\ep}_t \ \ov g'_\delta({\u^\ep}_x)\right](x,t) dx
 	=  \int_{I_i}\left[ \left(\ep  {\u^\ep}_{xx}
 	- \li {\u^\ep}_{x} \right)\ \ov g'_\delta({\u}_x)\right](x,t) dx
 $$
 whence 
 $$
   	\int_{I_i}  \li \left[{\u^\ep}_{x} \ \ov g'_\delta({\u^\ep}_x)\right](x,t) dx =
$$
 $$
 	= \ep  [{\u^\ep}_{x}
 	\ \ov  g'_\delta({\u^\ep}_x)] (L_i,t) -\ep [ {\u^\ep}_{x}
 	\ \ov  g'_\delta({\u^\ep}_x)] (0,t)
$$
$$
  	-\int_{I_i} \ep \left[{\u^\ep}_x \ \ov g''_\delta({\u^\ep}_x) {\u^\ep}_{xx}\right](x,t) dx
   	- \int_{I_i}\left [{\u^\ep}_t \ \ov g'_\delta({\u^\ep}_x)\right ] (x,t)dx\ .
$$
 
 Now, let $\delta$ go to zero and we use the dominated convergence theorem, 
  taking  into account that ${\u^\ep}_x, {\u^\ep}_{xx},{\u^\ep}_t \in L^1(I_i)$ and the properties of  $\ov g_\delta$ (see (\ref{g})). So we obtain
 \be
    \label{iii}\begin{array}{ll}
    \dsp   \int_{I_i}  \li \vert {\u^\ep}_{x}(t)\vert  dx \leq  \Vert  {\u^\ep }_{t} (t)\Vert_{L^1(I_i)} +
   \ep  \vert {\u^\ep}_{x}
    (L_i,t) \vert-\ep  \vert {\u^\ep}_{x}(0,t)\vert
    \\  \\ \dsp  \leq   \Vert  {\u^\ep }_{t} (t)\Vert_{L^1(I_i)} +
     \vert \ep {\u^\ep}_{x}(L_i,t) -\li \u^\ep(L_i,t) \vert  + \vert \li \u^\ep(L_i,t) \vert\\ \\
     \leq  \Vert  {\u^\ep }_{t} (t)\Vert_{L^1(I_i)}+ C^i_{0} \Vert ( \ep {\u^\ep}_x -\li \u^\ep)(t)\Vert_{W^{1,1}(I_i)} 
+ \vert \li \u^\ep(L_i,t) \vert\\ \\
	\leq C^i_{1} \left(\ep \Vert {\u}_x(t)\Vert_{L^1(I_i)} +\li  \Vert {\u} (t)\Vert_{L^1(I_i)} 
	+ \Vert {\u}_t(t)\Vert_{L^1(I_i)}
	\right)    + \vert \li \u^\ep(L_i,t) \vert
     \ ,\end{array}
 \ee
  where we used the equation satisfied by  $\u^\ep$;
 then, for small $\ep$, 
 \be \label {iiii}
  \dsp   \int_{I_i}  \vert {\u^\ep}_{x}(t)\vert  dx \leq
  C^i_{2} \left(
   \Vert {\u}(t)\Vert_{L^1(I_i)} 
	+ \Vert {\u}_t(t)\Vert_{L^1(I_i)}
	\right)    + \vert \li \u^\ep(L_i,t) \vert\ ,
 \ee
 where $C^i_{0},C^i_{1}, C^i_{2}$ depend on Sobolev constants and on $\li$.
 
For $i\in \mathcal O$, the claim easily follows from (\ref{iiii}), since   $\u^\ep(L_i,t)={\u}_0^\ep(L_i,t)$.
  
  As a consequence we obtain the $L^\infty$-bound for $\u^\ep$, $i\in\mathcal O$,
\be
      \label{linO}\Vert \u^\ep(t)\Vert _{L^\infty(I_i)} \leq C^i_{3}
       \left(\Vert  {\u^\ep }_{t} (t)\Vert_{L^1(I_i)}+\Vert \u^\ep(t)\Vert_{L^1(I_i)}
       + \Vert   {{u_i^\ep}_0}\Vert_{W^{1,1}(I_i)} 
        \right)
       \ ,\ \ t>0\ ,
\ee
 where $C^i_{3}$ depends on $\li$ and   Sobolev constants. 
 
 In order to obtain the estimate in the claim for the incoming arcs, we fix $\ep$ and  $$\ov U_{\ep}:=\dsp \max_{j\in\mathcal I} \max_{t\in [0,T]} [u^\ep_j(L_j,t)]^+\ ,\qquad \underline U_{\ep}:=\dsp - \max_{j\in\mathcal I} \max_{t\in [0,T]} [u^\ep_j(L_j,t)]^-\ .$$
  If $\ov U_{ \ep}>0$, let   $ k\in\mathcal I$ and $ \hat t\in[0,T]$ be such that 
 $u_{k}^\ep(L_{k}, \hat t)=\ov U_\ep$; then, arguing as in (\ref{iii}), we have 
   \be\label{u1}
     \begin{array}{ll}
     \dsp \left  \vert    \sum_{j\in \M}\alpha_{kj} u^\ep_j(N,\hat t)\right \vert  =\left\vert [\ep {u_k^\ep}_x -\lambda_k u_k^\ep](L_k,\hat t) \right\vert \\ \\
     \leq C^k_{0}\left( \Vert {u_k^\ep}_t (\hat t)\Vert_{L^1(I_k)} +
      \ep \Vert  {u_k^\ep}_x(\hat t) \Vert_{L^1(I_k)}+ \lambda_k\Vert u_k^\ep(\hat t)\Vert_{L^1(I_k)}\right) \end{array}
 \ee
whence
 \be\label{u2}
     \begin{array}{ll}
     \dsp
       \ov U_{\ep} \left(   \alpha_{kk}  +  \sum_{j\in \mathcal I, j\neq k }\alpha_{kj} \right) \leq 
       \alpha_{kk} \ov U_\ep +  \sum_{j\in \mathcal I, j\neq k }\alpha_{kj} u^\ep_j(L_j,\hat t) 
       =\sum_{j\in \mathcal I}\alpha_{kj} u^\ep_j(L_j,\hat t)
        \\ \\
          \leq \dsp \left \vert \sum_{j\in \mathcal O }\alpha_{kj} u^\ep_j(0,\hat t)\right\vert +
     C_{0}^k\left( \Vert {u_k^\ep}_t (\hat t)\Vert_{L^1(I_k)} +
      \ep \Vert  {u_k^\ep}_x(\hat t) \Vert_{L^1(I_k)}+ \lambda_k \Vert u_k^\ep(\hat t)\Vert_{L^1(I_k)}\right) \ ; \end{array}
 \ee
 now we use (\ref{anz}) to say that $\left(   \alpha_{kk}  +  \sum_{j\in \mathcal I, j\neq k }\alpha_{kj} \right)>0$, and  (\ref{linO}) to obtain
 $$
  	\ov U_{\ep}\leq C_4
	 \dsp\sum_{j\in\mathcal O\cup \{k\}}\left(\Vert  {u_j^\ep }_{t} (\hat t)\Vert_{L^1(I_j)}+\Vert u_j^\ep(\hat t)\Vert_{L^1(I_j)} 
    	+ \Vert   {u^\ep_{0j}}\Vert_{W^{1,1}(I_j)} \right) 
     	+ C^k_0  \ep \Vert   {u_k}_x^\ep(\hat t)\Vert_{L^{1}(I_k)},
$$
    where $C_4$ is a positive constant depending on Sobolev constants, on $\li$ and on the coefficients $\alpha_{ij}$ ($i,j\in\M$).

\noindent As for the lower bound, if $\underline U_\ep<0$
 let $ h\in\mathcal I$ and $ \tilde t\in[0,T]$ be such that 
 $u_{h}^\ep(L_{h}, \tilde t)=\underline U_{\ep}$;  arguing as above
  we obtain 
$$
	\underline U_{\ep}
 	\geq -C_5
 	\dsp\sum_{j\in\mathcal O\cup \{h\}}\left(\Vert  {u_j^\ep }_{t} (\tilde t)\Vert_{L^1(I_j)}+\Vert u_j^\ep(\tilde t)\Vert_{L^1(I_j)} 
    	+ \Vert   {u^\ep_{0j}}\Vert_{W^{1,1}(I_j)} \right) 
   	- C^h_0  \ep \Vert   {u_h}_x^\ep(\tilde t)\Vert_{L^{1}(I_h)},
$$
   where $C_5$ is a positive constant depending on Sobolev constants, on $\li$  and on the coefficients $\alpha_{ij}$
   ($i,j\in\M$).

So, we achieved  the   $L^\infty$-bound for $\u^\ep (L_i,t)$, $i\in\mathcal I$,
\be
      \label{linOO}
      	\begin{array}{ll}\Vert \u^\ep(L_i,\cdot)\Vert _{L^\infty(0,T)}
    	\dsp   \leq C_{6} \left( \ep \sum_{j\in\mathcal I}   \sup_{s\in[0,T]} \Vert   {u_j}_x^\ep (s)\Vert_{L^{1}(I_j)}
	+\suMj\Vert   {u^\ep_{0j}}\Vert_{W^{1,1}(I_j)} \right)
          \\ \\
       	+ C_{6}
     	\dsp   \sup_{s\in[0,T]} \sum_{j\in\mathcal M}\left(
    	 \Vert  {u_j^\ep }_{t} (s)\Vert_{L^1(I_j)}+\Vert u_j^\ep (s)\Vert_{L^1(I_j)} 
   	 \right)
    \ ,
   \end{array} \ee
  where $C_6$ depends on   Sobolev constants, on $\lambda_\iota$  and on the coefficients $\alpha_{\iota j}$ ($\iota,j\in\M$), which allows us to 
 control the quantity $ \vert \li \u^\ep(L_i,t) \vert$ in (\ref{iiii}) also when $i\in\mathcal I$, obtaining, for $t\in[0,T]$,
$$
  \begin{array}{ll}
	\dsp \int_{I_i}  \vert {\u^\ep}_{x} (x,t)\vert dx \leq 
	C^i_{2} \left(
	\Vert {\u}(t)\Vert_{L^1(I_i)} 
	+ \Vert {\u}_t(t)\Vert_{L^1(I_i)}
	\right)
     \\ \\
	\dsp +\li C_6 \left( \ep  \sum_{j\in\mathcal I}
	   \sup_{s\in[0,T]}\Vert   {u_j}_x^\ep ( s)\Vert_{L^{1}(I_j)} 
	 + \suMj \Vert   {u^\ep_{0j}}\Vert_{W^{1,1}(I_j)} \right)
      \\ \\
    	\dsp   +\li  C_{6}
     	\dsp \sup_{s\in[0,T]}  \sum_{j\in\mathcal M}\left(
    	 \Vert  {u_j^\ep }_{t} (s)\Vert_{L^1(I_j)}+\Vert u_j^\ep (s)\Vert_{L^1(I_j)}     \right)\ ;
   \end{array} 
$$
finally
$$
	\begin{array}{ll}
	\dsp  \sum_{i\in\mathcal I} \sup_{t\in[0,T]} \int_{I_i}  \vert {\u^\ep}_{x}  (x,t)\vert dx \leq  C_7 \sum_{j\in\mathcal M} \Vert   {u^\ep_{0j}}	 
	\Vert_{W^{1,1}(I_j)}
      \\ \\
	\dsp+C_7 \sup_{t\in[0,T]}  \sum_{j\in\mathcal M}\left(
     	\Vert  {u_j^\ep }_{t} ( t)\Vert_{L^1(I_j)}+\Vert u_j^\ep ( t)\Vert_{L^1(I_j)} \right)
   	\dsp + C_7 \ep \sum_{j\in\mathcal I}  \sup_{t\in[0,T]} 
    	\Vert   {u_j}_x^\ep ( t)\Vert_{L^{1}(I_j)}  , \end{array} 
$$
	where $C_7$ depends on Sobolev constants, $\lambda_\iota$, $\alpha_{\iota j}$ ($\iota,j\in\M$).
For small $\ep$ we have the claim.
\end{proof}

\medskip

Finally we prove the following comparison result, relying on the fact that 
the functions $U_i(x)=0$,  for all $x\in I_i$ and all $i\in\M$, satisfies the equations and the transmission conditions in (\ref{para}).

\begin{prop} \label{nonne}
Let $u^\ep$ be the solution to problem (\ref{para})-(\ref{uop});
if ${u^\ep_0}_i(e_i)\geq 0$ for all $i\in\M$,  then, for all $t>0$,
$$ \dsp \suM   \iIi  [ \ue_i(x,t)]^- dx  \leq \suM  \iIi  [ {\ue_0}_i(x)]^- dx \ ;$$
in particular, if ${u^\ep_0}_i(x)\geq 0$ then $u^\ep_i(x,t)\geq 0$ .
\end{prop}
\begin{proof}
Following standard methods (for example see  \cite{Daf}),  using  the function (\ref{gg}) we can write
 $$ 
     \dsp\suM  \left(  \iIi  g_{\delta}(- \ue(x,t)) dx -  \iIi  g_{\delta}( -\ue(x,0)) dx \right)  =
      \suM   \int_0^t \iIi  \left( g_{\delta}(- \ue)(x,s) \right)_s dx ds
 $$
  $$
     \leq \int_0^t    \sum_{i,j\in\M} \alpha_{ij} \ue_j (N,s)
     g'_\delta( -\ue_i(N,s)) ds
  $$
  $$
  + \suM  \int_0^t \iIi  \li g''_\delta( -{\ue}(x,s)) ( {\ue_i}(x,s))_x  {\ue_i}(x,s)
      dx\  ds \ .
  $$
  Since   $\ue_i(t), {\ue_i}_x(t)\in L^1(I_i)$ and $\ue_i(N)\in L^1(0,t)$ for all $i\in\M$, thanks to (\ref{g})
we can let $\delta$ go to zero 
using the dominated convergence theorem, to obtain 
$$ 
     \suM   \iIi  [ \ue_i(x,t)]^- dx  - \suM  \iIi  [ {\ue_i}(x,0)]^- dx 
$$
$$
      \dsp \leq \int_0^t \suM 
      \sum_{j\in\M} \alpha_{ij} |\ue_j (N,s)|   \sgn( {\ue_j}(N,s)) \sgn^+( -{\ue_i}(N,s)) ds  \leq 0
$$
 since the assumptions on the coefficients ensure that 
 $$ \suM   \alpha_{ij} \sgn( {\ue_j}(N,s)) \sgn^+( -{\ue_i}(N,s)) ds \leq 0\\ , \t{ for all } j\in\M  .$$  
 \end{proof}
\end{section}

\begin{section}
{\bf Convergence }
In this section we prove the vanishing viscosity approximation result.
The first subsection is devoted to the proof 
that, when  $\ep_n\to 0$, each  sequence $\{u^{\ep_n}\}$ of solutions to 
parabolic problems (\ref{para})-(\ref{uop}) has a subsequences converging 
to a solution $u$ of  (\ref{ipe})-(\ref{bci2}) satisfying (\ref{consflux}); this convergence result holds 
provided  $\{u_0^{\ep_n}\}$ 
approximates $u_0$ in the sense  of the  Lemma \ref{appr} below.
In Subsection 4.2 we prove that such limit function $u$
satisfies conditions 
(\ref{tci}), where  $\gamma_{ij}$ are  univokely determined by $\li$ and $K_{ij}$ ($i,j\in\M$), so that, thanks to uniqueness 
of solution to (\ref{ipe})-(\ref{bci2}), (\ref{tci}), each  sequence $\{u^{\ep_n}\}$  converges to $u$. 
 In particular, if the family  $\{K_{ij}\}$ verifies  (\ref{lc}), then   (\ref{coga2}) holds for the parameters $\gamma_{ij}$, i.e. the conditions satisfied  by the  limit function $u$ at the inner  node $N$ are   actually  transmission conditions.

\begin{subsection}{Vanishing viscosity limit}

Let us set 
\be
	\label {De}
      D_{\ep} :=\left \{ 
      \begin{array}{ll}
      w\in W^{2,1}(\A): \ \
       w_i(e_i)=B_i\ \t {if } i\in\mathcal I\ ,
     \\ \\
     \beta_{i}\left( \li  w_i(N) -\ep {w'_i}(N)\right)=\dsp  \sum_{j\in\M} \alpha_{ij} w_j(N) \ , \   
     i\in\M
     \end{array}
      \right\}\ ,
\ee
where $\beta_i$ is defined in (\ref{17}).

We will use the following lemma  whose quite technical proof is included in the Appendix.
\begin{lemma} \label{appr}
Let 
 $\{\ep_n\}_{n\in\N}$ be any decreasing sequence such that $\ep_n\to 0$.
 For every $v\in BV(\A)$ there exists a sequence
 $\{v_n\}_{n\in\N}$ converging to $v$ in $L^1(\A)$
 such that 
$v_n\in D_{\en}$ for each $n\in\N$, 
 $\Vert {v'_n} \Vert _{L^1(\A)}$ are  bounded by a quantity dependent on $\Vert v\Vert _{BV(\A)}$ and independent of $n$, and  $\ep_n \Vert v_n \Vert _{W^{2,1}(\A)} $ are bounded independently of $ n$.
\end{lemma}

Let  $u_0\in BV(\A)$; in this section, 
for each decreasing sequence  $\{\ep_n\}_{n\in\N}$  such that  $ \ep_n\to0$,  
$\{u_0^{\ep_n}\}_{n\in\N}$
denotes a sequence  
approximating  $u_0$ as in Lemma \ref{appr} .
 
 Let $ \{ u^{\ep_n}\}_{n\in\N}$ be the sequence of solutions to the following problems:
\begin{equation}
	\label {para2}\left\{ \begin{array}{ll}
	{{\u}_t^{\ep_n}}= -\li {\u}^{\ep_n}_x + \ep_n {\u}^{\ep_n}_{xx}\ , \qquad \qquad x\in I_i\ ,\ \ t\in[0,T]\ ,\qquad i\in\M ,
	\\ \\
	\u^{\ep_n}(x,0)={u_{0_i}^{\ep_n}}(x)
		\in W^{2,1}(\A)\  
		\ , \qquad   x\in I_i ,\ \qquad\qquad\qquad \ \ i\in\M ,
	\\ \\
		\beta_{i}  \left( -\li \u^{\ep_n}(N, t) +\ep_n{ \u}_x^{\ep_n}(N, t)\right)
	= \dsp \sum_{j\in\M} K_{ij} (u^{\ep_n}_j(N,t)-\u^{\ep_n}(N,t))\ , \quad t\in[0,T]\ ,\\ \\
	\u^{\ep_n}(0,t)=B_i
		\ ,\quad  i\in\mathcal I\ , \qquad\ 
	\u^{\ep_n}(L_i,t)={u_{0_i}^{\ep_n}}(L_i)\ ,\quad  i\in\mathcal O\ , 
		 \qquad\quad t\in[0,T]\ ,
	\end{array} \right.
\end{equation}
where $K_{ij}$ satisfy (\ref{17})  and $u_{0}^{\ep_n}\in D_{\ep_n}$.

The results in the previous section imply the following proposition.
\begin{prop}
Let (\ref{anz}) hold,
let 
$u_0\in BV(\A)$ 
and $B_i \in\R$
for $i\in\mathcal I$;
for all $T>0$ 
$$ \sup\{\Vert u^{\ep_n}\Vert_{BV(\A\times (0,T))}\}_{n\in\N}
< +\infty\ .$$

\end{prop}
\begin{proof} Let $C_1\geq \ep_n \Vert u^{\ep_n}_0 \Vert _{W^{2,1}(\A)} $ for all $\ep_n$.
Lemma \ref{appr} and Proposition \ref{stima1} imply 
$$\Vert u^{\ep_n}(t)\Vert_{L^1(\A)} \leq C_2$$
for $t\in[0,T]$, where $C_2$ depends on $T$ and on $\Vert u_0\Vert_{L^1(\A)}$;
Lemma \ref{appr} and Proposition \ref{stima3} imply 
$$\Vert u^{\ep_n}_t(t)\Vert_{L^1(\A)} \leq C_3$$
for $t\in[0,T]$, where $C_3$ depends on $ \Vert u_0\Vert_{BV(\A)}$ and on the quantity $C_1$;
finally, Propositions \ref{ux2}, \ref{stima1} and \ref{stima3}
imply 
$$\Vert u^{\ep_n}_x(t)\Vert_{L^1(\A)} \leq C_4$$
for $t\in[0,T]$, where $C_4$ depends on $T,  \Vert u_0\Vert_{BV(\A)}$, the quantiy $C_1$,
$\lambda_\iota$, $K_{\iota k}$ ($\iota,k\in\M$) and Sobolev constants.
The constants $C_i$, $i=2,3,4$ are independent from $\ep_n$, so the claim is proved.
\end{proof}

The first step in proving the   convergence result  is the argument of the next proposition.

\begin{prop}\label{propovisco}
$\ $Let (\ref{anz}) hold, let 
$u_0\in BV(\A)$,
 $B_i \in\R$
for  $i\in\mathcal I$. For all $T>0$ and for each decreasing sequence   $\{\ep_n\}_{n\in\N}$  such that 
 $ \ep_n\to0$, the  sequence $\{ u^{\ep_n}\}_{n\in\N}$  of solutions to (\ref{para2}) admits a subsequence 
converging in
$ L^1((\A\times (0,T))$ to a 
   solution  $u$ of (\ref{ipe})-(\ref{consflux}).   For all $i\in\M$,
   the corresponding  subsequence of traces 
   $\u^{\ep_{n_k}}(N, t)$ 
   converges weakly$\dsp ^*$ in $L^\infty(0,T)$ to some  functions $\mathcal W^{N}_i (t)$; besides, for all $i\in\mathcal O$,  the limit  function $u$ 
   satisfies 
   \be\label{bt}u_i(0,t)=\mathcal W^{N}_i (t) \ , \  t>0\ \ .  \ee
   \end{prop}

\begin{proof}

Thanks to the previous proposition, fixed $T>0$, for every  sequence $\{ u^{\ep_n}\}
$
($\dsp \ep_n\to 0$)  
there exist a subsequence $\{ u^{\ep_{n_k}}\}$ converging in
$ L^1(\A\times (0,T))$ to a
function $u\in BV(\A\times (0,T))$, for $n\to +\infty$.
It is possible   to consider its  traces 
$u(x,0)\in L^\infty(\A)$ and  $\u(0,t), \u(L_i,t) \in L^\infty (0,T)$ ($i\in\M$), and we know that
$$\dsp \lim_{ t\to 0} \Vert u(\cdot,t) -u(\cdot,0)\Vert_{L^1(\A)} = 0\ , $$
$$\dsp   \lim_{ x\to 0} \Vert \u(x,\cdot) -\u(0,\cdot)\Vert_{L^1(0,T)} \ , 
\lim_{ x\to L_i}   \Vert \u(x,\cdot) -\u(L_i,\cdot)\Vert_{L^1(0,T)} =0\  \ ,\ i\in\M,
$$
see \cite{BLN}.

We are going to identify the function $u$.
First we have,
for $i\in\mathcal I$, 
$$\begin{array}{ll} \dsp \int_0^T \iIi \left  [\u^{\ep_{n_k}} {\phi_i}_t + \li \u^{\ep_{n_k}} {\phi_i}_x - \ep_{n_k} {{\u}_x^{\ep_{n_k}}}{ \phi_i}_x  \right](x,t) \,dx dt  
+\iIi  {\u}^{\ep_{n_k}}_0(x){\phi_i}(x,0) dx\\ \\
 \dsp=\int_0^T  \phi_i (L_i,t) \suMj \alpha_{ij} u^{\ep_{n_k}}_j (N,t) dt\ ,\  \qquad \qquad \forall  \phi_i\in  \mathcal C^\infty_0((0,L_i]\times [0,T))
\ ,
\end{array}$$
and in similar way, for $i\in\mathcal O$, 
$$\begin{array}{ll} \dsp \int_0^T \iIi \left  [\u^{\ep_{n_k}} {\phi_i}_t + \li \u^{\ep_{n_k}} {\phi_i}_x - \ep_{n_k} {{\u}_x^{\ep_{n_k}}} {\phi_i}_x  \right](x,t)\, dx dt  
+\iIi  {\u}^{\ep_{n_k}}_0(x) {\phi_i}(x,0) dx\\ \\
\dsp=\int_0^T \phi_i (0,t) \suMj \alpha_{ij} u^{\ep_{n_k}}_j(N,t)
dt\ ,    \qquad\qquad   \forall \phi_i\in  \mathcal C^\infty_0([0,L_i)\times [0,T))
\ ;
\end{array}$$
here we have used the expression (\ref{cftc}) for the transmission conditions. 

First, we recall that, thanks to  Propositions   \ref{stima1},  \ref{stima3} and  \ref{ux2}, for $i\in\M$, 
$$ \dsp \int_0^T\iIi  \ep_{n_k} {{\u}_x^{\ep_{n_k}}} {\phi_i}_x dx dt \rightarrow_{k\to +\infty}  0
\ .$$
Then
we notice that,  since $u^{\ep_{n_k}}\in \mathcal C([0,+\infty);W^{2,1}(\A))
$,  thanks to Propositions \ref{stima1}$-$\ref{ux2}, 
 the sequence of the traces
 $\u^{\ep_{n_k}}(N, t)$  is  uniformly bounded in 
 $L^\infty(0,T)$, so that we can consider   a subsequence 
   converging weak$\dsp ^*$ in $L^\infty(0,T)$ to
a certain   function $\mathcal W^{N}_i (t),$ for each $i\in\mathcal M$.

 So, we can consider a subsequence of $\{\ep_{n_k}\}$, still denoted by $\{\ep_{n_k}\}$, and letting $\ep_{n_k}$ go to zero in the above equalities,
 using the dominated convergence theorem and Lemma \ref{appr}, we obtain

 \be\label{bordoep1}
 	\begin{array}{ll} \dsp 
	\int_0^T \iIi \left  [\u {\phi_i}_t + \li \u {\phi_i}_x   \right](x,t) \,dx dt  
	+\iIi  {\u}_0(x){ \phi_i}(x,0) (x)dx\\ \\
	\dsp=\int_0^T \phi_i (L_i,t) \suMj \alpha_{ij} \mathcal W^N_j(t)
		dt  \ , \qquad \forall  \phi_i\in  \mathcal C^\infty_0((0,L_i]\times [0,T))\ \ i\in\mathcal I
	\ ,
	\end{array}
 \ee
  \be
  	\label{bordoep2}
	 \begin{array}{ll} \dsp \int_0^T \iIi \left  [\u {\phi_i}_t + \li \u {\phi_i}_x   \right](x,t) dx dt  
	+\iIi  {\u}_0(x) {\phi_i}(x,0) dx \\ \\=
	\dsp\int_0^T \phi_i(0,t) \suMj \alpha_{ij} \mathcal W^N_j(t) dt  
		\ ,\qquad  \forall  \phi_i\in  \mathcal C^\infty_0([0,L_i) \times [0,T))\ ,\ \ i\in\mathcal O
	\ .
	\end{array}
 \ee
 
 Using test functions vanishing in $\{N\}\times [0,T)$, it immediately follows that  $\u$    is a weak solution to the conservation law (\ref{ipe}) and satisfies in weak sense the initial condition.

 In order to prove that $u$ satisfies the boundary  conditions  (\ref{consflux}) and (\ref{bt}), we 
are going to  follow the \cite{BLN, Ser}, where, however, only boundary conditions were considered. Here we adapt their arguments to transmission conditions.

Let 
$\phi_i\in  \mathcal C^\infty_0(\overline I_i\times [0,T))$; 
we set
\be
	\label{-1}
	\mathcal I[\phi_i]:= 
	\int_0^T \iIi [\u\phi_{i_t} +\li\u\phi_{i_x}](x,t)dx + \iIi  {\u}_0(x) \phi_i(x,0) dx\ .
\ee
Moreover, let $\zeta \in  \mathcal C^1_0(\R)$ be a suitable function such that $\zeta(x)=1$ in a neighborhood of $x=0$; we set $\zeta^a_i(x):= \zeta\left(\frac {d( x, \partial I_i)}{a}\right)$;  if   the parameter  $a$ is positive and  small,
 $(1-\zeta^a)\phi_i\in  \mathcal C^1_0(I_i\times [0,T))$, then
\be 
	\label{0}
	\mathcal I[\phi_i]= \mathcal I[(1-\zeta^a_i)\phi_i]+ \mathcal I[\zeta^a_i \phi_i]=\mathcal I[\zeta^a_i \phi_i]
	= \li \int_0^T  \left(   {\u} \phi_i(L_i,t)-{\u} \phi_i(0,t)\right) dt\ ,
\ee
where the last equality follows letting $a$ go to zero and using the properties of the traces \cite{Ser}.

As a consequence of the above computations, by density argument,  for  all  test functions 
$\phi_i\in W^{1,1}_0( \R\times  (-\infty,T))$ the following formula holds, for $i\in\M$, 
\be
	\label{1}
	\begin{array}{ll} \dsp  \int_0^T \iIi [\u\phi_{i_t} +\li \u\phi_{i_x}](x,t) dx dt& +  \dsp \iIi  {\u}_0(x) \phi_i (x,0) dx 
	\\
	&=\dsp \li  \int_0^T  \left(   {\u} \phi_i(L_i,t)-{\u} \phi_i(0,t)\right) dt\ .
	\end{array}
\ee

First, using (\ref{1}) with  test  functions 
$\phi_i$
having 
 null trace in  $\{e_i\}\times (0,T)$  (we recall that $e_i$ is the  the boundary node of the arc $I_i$), and (\ref{bordoep1}), (\ref{bordoep2}), we obtain 
\be \label{tct} \begin{array}{ll}\dsp \li \u (L_i,t) =\suMj \alpha_{ij} \mathcal W^N_j(t)\ \ \t{ a.e. in } (0,T)\ ,\ \  i\in \mathcal I\\
\dsp -\li \u (0,t) =\suMj \alpha_{ij}  \mathcal W^N_j(t)\ \ \t{ a.e. in } (0,T)\ ,\ \  i\in \mathcal O\ ;\end{array}\ee
the previous equalities prove that $u$ satisfies the conservation of the fluxes (\ref{consflux}) at the inner node $N$\ ,
since $\dsp \suM \alpha_{ij}=0$.

Then we use the technique  in  \cite{Ser}, considering at the inner node the functions $\mathcal W^N_i(t)$, $i\in\mathcal O$,
in place of  prescribed boundary functions .
We consider  functions $w_i\in W^{1,1}(I_i\times (0,T))$, $i\in\mathcal I$, whose traces in $\{0\}\times (0,T)$
are 
$B_i$ 
and functions $w_i\in W^{1,1}(I_i\times (0,T))$, $i\in\mathcal O$, whose traces  in $\{0\}\times (0,T)$  are $\mathcal W_i^N(t)$.

Let  $E(\xi)=\xi^2$, then  the following inequality holds
$$ E(\u^{\ep_{n_k}})_t \leq \ep_{n_k} (E(\u^{\ep_{n_k}}))_{xx}- \li (E(\u^{\ep_{n_k}}))_x\ ,\ \ i\in\M\ ,$$
and considering positive test functions $\phi_i$ vanishing  in  $\{L_i\}\times (0,T)$, we obtain
$$
0\leq \int_0^T \iIi\left[E(\u^{\ep_{n_k}}) \left( {\phi_i}_t +  \li{\phi_i}_x\right) -\ep_{n_k} (E(\u^{\ep_{n_k}}) )_x {\phi_i}_x\right](x,t)dx dt$$
$$+
\iIi E({u_0}^{\ep_{n_k}}_i (x))\phi_i (x,0) dx
 +\int_0^T \left[ \left(\li E(\u^{\ep_{n_k}}) -{\ep_{n_k}} (E(\u^{\ep_{n_k}})) _x\right) {\phi_i}\right] (0,t)
dt
\ ,$$
for all $i\in\M$; moreover, using the equation satisfied by $\u^{\ep_n}$ and  test functions 
$\psi_i=\phi_i E'(w_i)$, we have
$$0= \int_0^T \iIi \left[\u^{\ep_{n_k}} \left( {\psi_i}_t +  \li{\psi_i}_x\right) -\ep_{n_k} {\u}^{\ep_{n_k}}_x {\psi_i}_x\right] (x,t)dx dt$$
$$+
\iIi {u_0}^{\ep_{n_k}}_i(x) \psi_i(x,0) dx
 +\int_0^T \left[\left(\li \u^{\ep_{n_k}} -  \ep_{n_k} {\u}^{\ep_{n_k}} _x\right) {\phi_i} E'(w_i)\right](0,t)
dt
\ .$$
Subtracting the last  relation from the previous one we obtain
$$\begin{array}{ll}
	\dsp 0\leq 
	\int_0^T \iIi\left[ E(\u^{\ep_{n_k}})  {\phi_i}_t- \u^{\ep_{n_k}}{\psi_i}_t
	+  \li  \left( E(\u^{\ep_{n_k}})  {\phi_i}_x- \u^{\ep_{n_k}}{\psi_i}_x\right)\right](x,t) dx dt
\\
\dsp -  \int_0^T \iIi\ep_{n_k} \left[ (E(\u^{\ep_{n_k}})) _x {\phi_i}_x - {\u}^{\ep_{n_k}}_x {\psi_i}_x \right](x,t) dx dt
      \\
\dsp+
\iIi  \left(E({u_0}^{\ep_{n_k}}_i(x) )-E'({w_i(x,0)}){\u}^{\ep_{n_k}}_0(x)\right)\phi_i(x,0) dx
\\ \dsp  +
\int_0^T \left[ \left( \li \left(E(\u^{\ep_{n_k}}) -\u^{\ep_{n_k}} E'(w_i)\right) 
-
\ep_{n_k}
 \left
(E'(\u^{\ep_n})   - E'(w_i) \right) {\u}^{\ep_{n_k}}_x \right)
{\phi_i}  \right](0,t) dt.
\end{array}$$
Now we let $k\to +\infty$ and using the dominated convergence theorem and taking into account that  
$ \Vert {u^{\ep_n}}(t)\Vert_{W^{1,1}(\A)} $ and 
$\ep_n \Vert {u^{\ep_n}}(t)\Vert_{W^{2,1}(\A)} $  are bounded in $[0,T]$, uniformly in $n$, we obtain 
$$\
	\begin{array}{ll}
	\dsp 0\leq 
	\int_0^T \iIi\left[ E(\u)  {\phi_i}_t- \u{\psi_i}_t
	+  \li  \left( E(\u)  {\phi_i}_x- \u{\psi_i}_x\right) \right](x,t)dx dt
      \\ 
	\dsp  +\iIi \left(E({u_0}_i (x))-E'({w_i(x,0)}){\u}_0(x)\right)\phi_i (x,0) dx
	\\
	+\li \dsp \int_0^T  \left[\left( E(w_i) - w_i E'(w_i)\right) {\phi_i} \right](0,t) dt \ ,
	\end{array}
$$
then  (\ref{1}) implies
$$
	0\leq 
	\int_0^T \iIi\left[ E(\u)  {\phi_i}_t
	+  \li   E(\u)  {\phi_i}_x
      	\right] (x,t) dx dt + \iIi E({u_0}_i(x) )\phi_i(x,0) dx
$$
$$ 
	+\li 
	\int_0^T 
	\left[ \left(E(w_i) + E'(w_i) (\u-w_i)\right) {\phi_i}\right] (0,t) dt\ ,\qquad \t{ for all } i\in\M \ .
$$

Now, we consider  positive functions $\eta_i\in  \mathcal C^1_0((-\infty,L_i)\times [0,T))$ and   the  functions
$\zeta_a$ previously introduced in this proof, with  $0\leq \zeta\leq 1$; 
in the above inequality  we choose $\phi_i=\phi_i^a:=\eta_i \zeta_i^a$ .

As $a$ goes to zero, as in \cite{Ser},
we obtain
$$
	0\leq  
		 \li \int_0^T\left[ \left(  E(w_i) -E(\u)
	+E'(w_i)   (\u-w_i)\right) {\eta_i} \right] (0,t) dt,  \qquad \t{ for all } i\in\M\ ;
$$
since the above inequality holds for all positive $\eta_i\in  \mathcal C^1_0((-\infty,L_i)\times [0,T))$, it follows that, for all $i\in\M$,
$$\u^2(0,t)-w_i^2(0,t)  \leq 2w_i(0,t)  (\u(0,t)-w_i  (0,t)) \  , \t{ a.e. } t\in (0,T)\ .$$
It is readily seen that 
 the above relation must be an equality 
and gives $u_i(0,t)=w_i(0,t)$ for a.e. $t$, i.e.
\be
	\label{ccbb}
	\u(0,t)=B_i\ , \ i\in\mathcal I\ , \\ \  \ \ \ \  \ \u(0,t)=\mathcal W^N_i(t)\ , \ i\in\mathcal O\ \ ,\qquad \t{ a.e. } \t{ in } (0,T)\ .
\ee
\end{proof}

We remark that the conditions (\ref{ccbb}) identify each limit  function $u_i$ as the unique solution 
to equation (\ref{ipe}) 
with boundary conditions in  $x=0$ given by (\ref{ccbb}) (see Definition \ref{sol} and (\ref{1}) with $ \phi_i(L_i,t)=0$) .
\end{subsection}
\begin{subsection}{Identification of the limit}

Now we are going to prove that all the sequences $\{u^{\ep_n}\}$ converge to the same   limit function, showing that 
the limit function $u$ in the  Proposition \ref{propovisco}  does not depend on the particular subsequence.

First we notice that the limit function  $u$ is univoquely determined  on the incoming arcs $ I_i$, $i\in\mathcal I$,  by the boundary and initial conditions  for these  arcs;
moreover, we recall that the limit function $u$ satisfies   the equalities (\ref{tct}), which, 
taking into account  the second equalities in  (\ref{ccbb}), 
can be written in the following way
\be
	\label{tcf}\begin{array}{ll} \dsp\sum_{j \in \mathcal I} \alpha_{ij} \mathcal W_j^N(t)
		+
	\sum_{j \in \mathcal O} \alpha_{ij} u_j(0,t)
		=\li \u(L_i,t) \ , \ \  \qquad \qquad \qquad   i\in\mathcal I\ ,\\
	\dsp \sum_{j \in \mathcal I} \alpha_{ij} \mathcal W_j^N(t)
		+( \alpha_{ii} +\li)  u_i(0,t)
		+
\sum_{j \in \mathcal O,j\neq i } \alpha_{ij} u_j(0,t)
=0
\ \ , \  i\in\mathcal O\ 
,\end{array}\ee
for  $a.e.\  t\in(0,T)$; using these equalities we are going to prove that the values $u_j(0,t)$, $j\in\mathcal O$, are determined by the values $u_j(L_j,t)$,  $j\in\mathcal I$, by (\ref{tci}) where the parameters $\gamma_{ij}$ depend only on $\alpha_{ij}$ 
(i.e. $K_{ij}$) and $\lambda_i$.

 Let   $Q$  be the $m\times m$ coefficients matrix 
of the linear system (\ref{tcf}) for the unknowns $\mathcal W_\iota^N(t), u_l(0,t)$, $\iota\in\mathcal I, l\in\mathcal O$; thanks to (\ref{17}) and (\ref{aij}) this matrix has some useful properties we are going to prove.

We assume that $\mathcal I=\{1,2,...,m_\mathcal I\}$ and $\mathcal O= \{m_\mathcal I+1,m_\mathcal I+2,...,m\}$ and we set
$m_\mathcal O= m-m_\mathcal I$.
The matrix $Q
$ has the form

 $$	 Q=\begin{pmatrix}
  	\alpha_{11} & 
		 .\ &. \ &
	 \alpha_{1{m_\mathcal I}} & 
	 \alpha_{1 {m_\mathcal I}+1} &
	   .\ &.\ &  
	   \alpha_{1m}
	       \\
         . \ & . \ & . \ & . \ & .\ &.\ &.\ &.
     \\  
          . \ & . \ & . \ & .\ & .\ &.\ &.\ &.
    \\
           \alpha_{m_\mathcal I 1 }&
		 .\ &. \ &
	 \alpha_{m_\mathcal I m_\mathcal I}  &
	  \alpha_{m_\mathcal I m_\mathcal I+1}   &
	  .\ &. \ &
	    \alpha_{m_\mathcal I m}
	    \\ 
\\  
          \alpha_{m_\mathcal I+1 1 } &
                 .\ &. \ &
          \alpha_{m_\mathcal I+1 m_\mathcal I }&
         \alpha_{m_\mathcal I+1  m_\mathcal I+1}+\lambda_{m_\mathcal I+1}  &
                  .\ &. \ &  
                    \alpha_{m_\mathcal I+1m}
             \\ 
         . \ & . \ & . \ & . \ & .\ &. \ &.\ & .
     \\  
          . \ & . \ & . \ & .\ & .\ &.\ &.\ & .
    \\
	\alpha_{m1 } &
	 .\ &. \ &
	 	 \alpha_{mm_\mathcal I}& 
	  \alpha_{mm_\mathcal I+1} &
	   .\ &. \ &
	 \alpha_{mm}+\lambda_{m}  
	   	 \end{pmatrix}  \ .
  $$
\medskip

First  we prove that  the assumptions on $K_{ij}$, $i,j\in\M$ (i.e. on $\alpha_{ij}$, see (\ref{aij})),  imply that  
the    matrix 
$Q$ is nonsingular.  
We need some  preliminary definitions and theorems, for whom we refer to \cite{BCM, House}. 

\begin{defi} Given a  matrix $P\in \C^{n\times n} $,  
$P=\{p_{ij}\}$,
and $n$ points $Y_i$ in the plane, 
the oriented graph associated to $P$ is the graph obtained joining the points $Y_i$ and $Y_j$ with an oriented arc  from 
$Y_i$ to $Y_j$, for all $i,j$ such that $p_{ij}\neq 0$.
\end{defi}

\begin{defi} An oriented graph is strongly connected if any two nodes are connected by an oriented walk ( i.e.
a sequence 
of oriented arcs $V_i$ and points $Y_i$, such that $V_i=(Y_i,Y_{i+1})$).
\end{defi}

We are going to deal with the class of  irreducible matrices. The definition of irreducible matrix can be found, for example, in \cite{BCM,House}; here, in order to apply the Theorem \ref{ger} below,  we are going to use the following characterization \cite{BCM,House}.
\begin{prop} $P\in \C^{n\times n} $ is an irreducible matrix if and only if its associated oriented graph  is strongly connected.
\end{prop}

\begin{theo}\label{ger}{ (Gershgorin  theorems)} $  $
Let $P\in \C^{n\times n} $, $P=\{p_{ij}\}$,  and let 
$\dsp J^P_{ i}= \{z\in\C: \vert z-p_{ii}|\leq \sum_{ j=1, j\neq  i}^n |p_{ ij}|\}$; then
all the eigenvalues of $P$ belongs to the 
set $\dsp \cup_{i=1}^n J^P_i$. 
Moreover, 
if $P$ is  an irreducible  matrix and $\mu$   is an eigenvalue 
 lying on the boundary of each disk $J^P_i$ which contains it, then it lies on the boundaries of all disks $J^P_i$, $i=1,...,n$. 

\end{theo}

Assume condition (\ref{anz}). We notice that we can consider the case when the parameters $K_{ij}$ (and, consequently, $\alpha_{ij}$) are such that the    matrix 
$Q$   is  irreducible.  If not, problem (\ref{para2}) can be splitted in two or more independent transmission problems, and for  each one of them the  corresponding coefficients matrix in system (\ref{tcf}) is  an irreducible matrix; so,   Theorem \ref{viscolim} below can be proven separately for each independent problem. 

\begin{lemma} Let (\ref{anz}) hold.
The matrix $Q$ is  nonsingular and $\det Q>0$.
\end{lemma}
\begin{proof}
The matrix $Q$ is  simmetric  and 
$$0< \alpha_{ii}=\sum_{j\in\M, j\neq i} \vert \alpha_{ij}\vert\ \t{ if } i\in \mathcal I\ ,\qquad \qquad \ 
\alpha_{ii}+\li>\sum_{j\in\M, j\neq i} \vert \alpha_{ij}\vert\ , \t{ if } i\in \mathcal O\ ,$$ 
thanks to (\ref{aij}), since $\li>0$ for $i\in\M$;
these facts and Theorem \ref{ger} imply that $Q$ has real positive eigenvalues since $J^Q_i\subseteq \{\Re z\geq 0\}$ for $i\in\M$, and  none eigenvalue can be zero since the origin does not belong 
to the disks $J^Q_i$ for $i\in\mathcal O$.
\end{proof}

The matrix $Q$ is a $M-$matrix, according with the following definition  \cite{Pl}.

\begin{defi}
A matrix $P\in \R^{n\times n}$ which can be expressed in the form $P = \sigma I- \ov P$, where $\ov P = \{\ov p_{ij}\} $
with $\ov p_{ij} \geq 0$, 
$1 \leq i, j \leq  n$, and $\sigma \geq \rho(\ov P)$, the maximum of the moduli of the eigenvalues of $\ov P$, is called an $M-$matrix.
\end{defi}

It is easy to check that the matrix $Q$ verifies the above definition with the position  
 $\sigma=\max \{\alpha_{ii}, \alpha_{jj}+\lambda_j: i\in\mathcal I, j\in\mathcal O\}$.

Non singular $M-$matrices have several properties; in particular
all their principal minors are positive and their  inverse matrices have non negative elements \cite{Pl}.

\medskip

In the following lemma we prove  further properties for the elements of the matrix $Q^{-1}$.
 
\begin{lemma} Let (\ref{anz}) hold and let $Z=Q^{-1}=\{z_{ij}\}_{i,j\in\N}$. 
For all $i\in\M$, $z_{ii}>0$; for $i\neq j$, if $\alpha_{ij}< 0$ then $z_{ij}>0$  ($i,j\in\mathcal M$).
\end{lemma}
\begin{proof}
Let consider  a submatrix
  of Q obtained by deleting a set of corresponding rows and columns 
  $$\begin{pmatrix}
  			\alpha_{k_1k_1}&
		 .&.&
		 \alpha_{k_1k_n}\ & \alpha_{k_1h_1}&
		   .&.&
		\alpha_{k_1h_l}\ 
	 \\ 
		.  \ & . \ & . \ & .\ &. \ & .\ & .& . \ &
	\\
		 . \ & . \ & . \ & .\ &. \ & .\ & .& . \ &
	\\
		\alpha_{k_nk_1 }\ & 
		.\ &.\ &
		 \alpha_{k_nk_n}&
		 \alpha_{k_nh_1}  &
		  .\ &.\ & 
		   \alpha_{k_n h_l}
	 \\  \\
		\alpha_{h_1k_1 }\ & 
		.\ &.\ &
		 \alpha_{h_1k_n}&
		 \alpha_{h_1h_1} +\lambda_{h_1} &
		  .\ &.\ & 
		   \alpha_{h_1 h_l}
	\\
	    	 . \ & . \ & . \ & .\ &.\ & .\ &.\ &.
	\\ 
	    	 . \ & . \ & . \ & .\ &.\ & .\ &.\ &.
	 \\  	 
	 	\alpha_{h_lk_1} \ & . \ & .\ &
	 	 \alpha_{h_lk_n} & 
	 	\alpha_{h_lh_1} &  
	 	 .\ &  .\ &
	 	\alpha_{h_lh_l}+\lambda_{h_l}\  
  	 \end{pmatrix}  \ ,
$$
where 
 \be
  \label{rckh1}\begin{array}{ll}  
   	0\leq n\leq m^\mathcal I\ , \qquad
	0\leq l\leq m^\mathcal O\ ,  
   \\ 
  	k_i\in \mathcal I\ \t{ for } i=1,2,...,n\ ,\qquad   h_i\in \mathcal O \t{ for } i=1,2,...,l\ ,
   \\
    	k_i\neq k_j\ \t{ for }  i,j=1,2,...,n\ ,\qquad  
   	h_i\neq h_j \t{  for } i,j=1,2,...,l \ ;
\end{array}\ee
  then we consider the  submatrix obtained edging the above one with   the  r$-th$ row and   the c$-th$ column of $Q$,
$$  
 	  H^{rc}_{nl}=\begin{pmatrix}
  	\alpha_{rc}\ & \alpha_{rk_1} &
	.&.&
	 \alpha_{rk_n} & \alpha_{rh_1} &  .&.&  \alpha_{rh_l} 
 	\\ \\
		\alpha_{k_1 c }\  & 
		 \alpha_{k_1k_1}&
		 .&.&
		 \alpha_{k_1k_n}\ & \alpha_{k_1h_1}&
		   .&.&
		\alpha_{k_1h_l}\ 
	 \\ 
		.\ & . \ & . \ & . \ & .\ &. \ & .\ & .& . \ &
	\\
	.	\ & . \ & . \ & . \ & .\ &. \ & .\ & .& . \ &
	\\
		\alpha_{k_nc}\ & 
		\alpha_{k_nk_1 }\ & 
		.\ &.\ &
		 \alpha_{k_nk_n}&
		 \alpha_{k_nh_1}  &
		  .\ &.\ & 
		   \alpha_{k_n h_l}
	 \\  \\
	 	\alpha_{h_1c}\ & 
		\alpha_{h_1k_1 }\ & 
		.\ &.\ &
		 \alpha_{h_1k_n}&
		 \alpha_{h_1h_1} +\lambda_{h_1} &
		  .\ &.\ & 
		   \alpha_{h_1 h_l}
	\\
	    	. \ & . \ & . \ & . \ & .\ &.\ & .\ &.\ &.
	\\ 
	    	. \ & . \ & . \ & . \ & .\ &.\ & .\ &.\ &.
	 \\  
		 \alpha_{h_lc}  & 
	 	. \ & . \ & .\ &
	 	 \alpha_{h_lk_n} & 
	 	\alpha_{h_lh_1} &  
	 	 .\ &  .\ &
	 	\alpha_{h_lh_l}+\lambda_{h_l}\  
  	 \end{pmatrix}  \ ,
  $$
  where 
  \be
  \label{rckh2}\begin{array}{ll}
  	r,c \in\M\ ,    \quad
  	   r\neq c\ ,\quad 
   	 	k_i,h_j \neq r,c \ \t{ for }i=1,2,...,n \ \t{ and } \ j=1,2,...,l.
\end{array}\ee
First  we are going to show  that all the $H^{rc}_{nl}-$type matrices  have non positive determinant, and 
that, if $\alpha_{rc}< 0$, then $\det H^{rc}_{nl}<0$.
 
   This fact is readily seen  
   for all $n,l$ such that  $n+l=0, 1,2$ (for any $r,c$ as above), using  (\ref{aij}) and the positivity of the principal minors of $Q$.

  In order to use  the principle of induction to prove that $\det H^{rc}_{nl} \leq 0$ 
  for all $l+n<m_{\mathcal I}+m_{\mathcal O}$ 
  and $r,c,k_i,h_j$ satisfying 
  (\ref{rckh1}), (\ref{rckh2}),
  we assume that 
  $\det H^{rc}_{nl} \leq 0$
   for all $n,l$ such that $n+l=\nu-1<m_{\mathcal I}+m_{\mathcal O}-2$  (and any $k_i,h_j,r,c$ as in  (\ref{rckh1}), (\ref{rckh2})), 
    and we compute $\det H^{rc}_{nl}$
  when  $n+l=\nu$  (and any $k_i,h_j,r,c$ as in  (\ref{rckh1}), (\ref{rckh2})):
  $$\det H^{rc}_{nl}=\alpha_{rc} \det  M_{rc} +
   \sum_{j=1}^{n} \alpha_{rk_j} (-1)^j(\det M_{rk_j})
    +
   \sum_{j=1}^{l} \alpha_{rh_j} (-1)^{n+j}(\det M_{rh_j})\ ,
$$
where $M_{rc}$  is the  matrix  obtained removing the first line and the first 
column in $H^{rc}_{nl}$, 
$M_{rk_j}$ is the one obtained removing the first line and the $j+1$-th
column  and $M_{rh_j}$ is the one obtained removing the first line and the $n+j+1$-th
column .
$M_{rk_1}$ is a $H^{k_1c}_{n-1l}$ matrix,  while, for all $j=2,...,n$,  $M_{rk_j}$ becomes a $H^{k_jc}_{n-1l}$ matrix 
provided $j-1$ exchanges of rows, and 
$M_{rh_j}$ becomes a $H^{h_jc}_{nl-1}$ matrix provided
 $n+j-1$ exchanges of rows. 
 
 Using the inductive assumption  we have
$$\det H^{rc}_{nl}=\alpha_{rc} \det  M_{rc}$$
$$ 
     -\sum_{j=1}^{n} \alpha_{rk_j} (-1)^j (-1)^{j-1}
        |\det M_{rk_j}|
    -\sum_{j=1}^{l} \alpha_{rh_j} (-1)^{n+j} (-1)^{n+j-1}
        |\det M_{rh_j}|\ 
$$
$$
	=\alpha_{rc} \det  M_{rc}
		 +\sum_{j=1}^{n} \alpha_{rk_j} 
       |\det M_{rk_j}|
    +\sum_{j=1}^{l} \alpha_{rh_j} 
        |\det M_{rh_j}|\ ;
$$
$\det M_{rc}>0$,  since it is a principal minor of $Q$,
so, thanks to (\ref{aij}), 
the  principle of induction  proves that $\det H^{rc}_{nl}\leq 0$
for all $r,c,n,l$ satisfying (\ref{rckh1}),(\ref{rckh2}). Finally, it is readily seen that $ \alpha_{rc}< 0$ implies 
$\det H^{rc}_{nl}<0 $.

Knowing the  sign of the  principal minors of  $Q$  and of the determinant of $H^{rc}_{nl}$ - type matrices  
 is the argument to prove the claim.

When $m_{\mathcal O}=m_{\mathcal I}=1$ it is readily seen that 
$ z_{ij}>0$ for $i,j=1,2$.
In general cases, let $Q_{ij}$ be the adjoint matrix to the element $q_{ij}$ (remember that $Q$ is simmetric):
\begin{itemize}
\item
{${z}_{ii}$:}
$${z}_{ii}= 
(\det Q)^{-1}\det Q_{ii}> 0\ ,$$
since the principal minors of $Q$ are positive;

\item{${z}_{ij} \ , i< j$:}
$${z}_{ij}= 
(\det Q)^{-1} (-1)^{i+j}\det Q_{ij}\ ,$$
where  :
\begin{itemize}
\item [-] $Q_{12}$ reveals to be  a $H^{21}_{(m_\mathcal I-2)m_\mathcal O}$ matrix  if $m_\mathcal I\geq 2$ 
 and a   $H^{21}_{0 (m_\mathcal O-1)}$ matrix  if $m_\mathcal I=1$ ; 
 $\det Q_{12}< 0$ if $\alpha_{12}<0$ ;
 
 \item [-]  provided $ j-2$ exchanges of rows,
$Q_{1j}$ becomes   a 
$H^{j1}_{nl}$ matrix, where $n+l=m_\mathcal I+m_\mathcal O -2$ ; $\det Q_{1j}\neq 0$ if $\alpha_{1j}<0$;
  \item [-] 
in general,  provided $ j-2$ exchanges of rows  and 
$i-1 $ exchanges of columns, $Q_{ij}$ becomes  a $H^{ji}_{nl}$ matrix, where 
$n+l=m_\mathcal I+m_\mathcal O-2$; $\det Q_{ij}\neq 0$ if $\alpha_{ij}<0$;
\end{itemize}
so 
$${ z}_{ij}= 
	(\det Q)^{-1} (-1)^{i+j}\det Q_{ij}
$$
$$
	=(\det Q)^{-1} (-1)^{i+j}  (-1)^{j-2+i-1}(-|\det Q_{ij}|)=
	(\det Q)^{-1} |\det Q_{ij}| \ 
$$
which implies $z_{ij}>0$ if $\alpha_{ij}<0$.

\medskip

\item{${ z}_{ij} \ , i> j$:}

the result follows by the simmetry of $Z$.

\end{itemize}

\end{proof}

In the following theorem we prove our main convergence  result .

\begin{theo} \label{viscolim}
Let (\ref{anz}) hold, let 
$u_0\in BV(\A)$
and $B_i \in\R$
for  $i\in\mathcal I$ .
There exist parameters $\gamma_{ij}$, satisfying (\ref{coga}), univokely determined by  $\li$ and $K_{ij}$ ($i,j\in\M$),
such that 
all the sequences $\{ u^{\ep_n}\}_{n\in\N}$ of solutions to problems (\ref{para2}) ($\ep_n \to 0$)
converge  in
$ L^1((\A\times (0,T))$ to 
 the  solution of (\ref{ipe})-(\ref{tci}), for all $T>0$.  
 If (\ref{lc}) holds, then the parameters $\gamma_{ij}$ satisfy (\ref{coga2}).
\end{theo}

\begin{proof}
Let $\ep_n\to 0$ and let $u$ be the limit function in Proposition \ref{propovisco}; on each arc $I_i$  incoming in the node  the  function $u_i$ is univokely 
determined by the initial data ${u_0}_i$ and the  boundary ones $B_i$, thanks to the first equalities (\ref{ccbb}) (see Section 1).

The matrix $Q$ of the system (\ref{tcf}) is nonsingular, then  $\mathcal W_i^N(t)$ and $u_j(0,t)$, ($i\in \mathcal I$,  
$j\in \mathcal  O$) are univokely determined by 
$\li \u(L_i,t)$, $i\in\mathcal I$;  in particular,
\be\label{ga}\lambda_j u_j(0,t) = \sum_{i\in\mathcal I} \lambda_jz_{ji} \lambda_i \u(L_i,t)\ , \ \ j\in\mathcal O\  ,\ee
where $Z=Q^{-1}=\{z_{ij}\}_{i,j\in\N}$;
so, 
on each outgoing arc $I_j$
the limit function $u_j$ is univokely 
determined by the initial datum ${u_0}_j$, by the initial data  ${u_0}_i$ and the  boundary ones $B_i$, for $i\in\mathcal I$, (see Section 1).

 This fact prove  that 
  the set of  possible limit functions for  subsequences $\{u^{\ep_{n_k}}\}$ contains only the solution to problem 
  (\ref{ipe})-(\ref{bci2}), satisfying (\ref{tci})  with  $\gamma_{ij}=  \lambda_iz_{ij}$ ($i\in\mathcal O$, $j\in\mathcal I$),
   in the sense of Definition \ref{sol}.
  
  The elements of $Z$ are non negative, since it is the inverse of a $M-$ matrix, 
    so $\gamma_{ij}\geq 0$; moreover,  equalities (\ref{tct}) and (\ref{ga}) easily imply that  $\dsp \sum_{i\in\mathcal O} \gamma_{ij}=1$.
    Finally, thanks to the previous lemma, if $\alpha_{ij}<0$ ($i\in \mathcal O$, $j\in \mathcal I$), then $z_{ji}>0$, so that condition
    (\ref{lc}) implies condition (\ref{coga2}).
 \end{proof}
 \end{subsection}
  
  \end{section}

\begin{section}{Approximation examples }

 In the previous section we proved that, if the  family of parameters $K_{ij}$ appearing  in the transmission conditions in (\ref{para2}) satisfies the constrains (\ref{17}),(\ref{anz}),(\ref{lc}), then,  
when $\ep_n \to 0$,  problems (\ref{para2})  approximate the first order transport problem (\ref{ipe})-(\ref{tci}) with appropriate transmission coefficients $\gamma_{ij}$ satisfying (\ref{coga}),(\ref{coga2}) (in the sense of Theorem \ref{viscolim}).
On the other hand, in the general case, we are unable to prove that,
given some parameters  $\gamma_{ij}$ satisfying (\ref{coga}),(\ref{coga2}),
it is possible to pick out some corresponding
coefficients $K_{ij}$ satisfying (\ref{17}),(\ref{anz})
(or, equivalently, $\alpha_{ij}$ satisfying   (\ref{aij}))
and to build a sequence of  linear parabolic problems as (\ref{para2}) approximating   the 
 problem (\ref{ipe})-(\ref{tci}). The main difficulty is in inverting some complicate matrices and
 this involves heavy computations. 
However, in  this section we prove such result 
for some particular and quite general instances.

\begin{subsection} {}

First we show that, when the transmission conditions (\ref{tci}) have the particular form
\be\label{nnp1}  \lambda_{i} u_i(0,t)= \gamma_{i}\sum_{j\in\mathcal I} \lambda_{j} u_j(L_j,t) 
\  \ \ \forall i\in \mathcal O ,\qquad  \gamma_{i}>0,\ \sum_{i\in\mathcal O} \gamma_i=1 \ ,\ee
it is possible to find families  $\{K_{ij}\}$ satisfying (\ref{17}), (\ref{anz})
in such a way  the limit $u$ of the sequence of solutions of problems (\ref{para2}) satisfies conditions (\ref{nnp1}).

 Let  $\mathcal I=\{1,2,...,m_\mathcal I\}$, $\mathcal O=\{m_\mathcal I+1,m_\mathcal I+2,...,m\}$;
we consider the following  parabolic transmission conditions , which are  particular cases of the ones in (\ref{para}),  
$$
   -\li \uei(N, t) +\ep {\uei}_x(N, t)
	= \dsp 
	\sum_{j\in \mathcal O} k_{j} (\ue_j(N,t)-\uei(N,t))  \ \ i\in\mathcal I\ ,
$$
$$
   \li \uei(N, t) -\ep {\uei}_x(N, t)
	= \dsp  k_{i} 
	\sum_{j\in\mathcal I}(\ue_j(N,t)-\uei(N,t))  \ \ i\in\mathcal O\ ,
$$
where $k_j>0$ ($j\in \mathcal O$);
this kind of   transmission conditions  involves  only the jumps between the solutions on  each ougoing arc and the  solutions  on each incoming one.

The corresponding coefficient matrix $Q$ of the linear system (\ref{tcf}) is  

$$
\quad  Q=\begin{pmatrix} 
 & & & &   -k_{m_\mathcal I+1} & -k_{m_\mathcal I+2} & ...&-k_{m} \\
	 & &
		Q_{11}
	&  &.&.&.&.&
	\\
	 & & &  &.&.&.&.\\
	 & & &  &     -k_{m_\mathcal I+1} & -k_{m_\mathcal I+2} & ... &-k_{m}
	\\
	-k_{m_\mathcal I+1} & .. &..&-k_{m_\mathcal I+1} & &   &  &
	\\	
	.&.&.&. & & Q_{22} & & 
	\\
	.&.&.&.&  &
	   &  &
	\\
	 -k_{m} & ..& ..&  -k_{m}  & & &  
	 \end{pmatrix}  ,
$$
where $Q_{11}, Q_{22}$ are diagonal matrices, 
$$\ Q_{11}=\left(\dsp \sum_ {j\in \mathcal O} k_{j} \right) I\ \ , 
\ \ Q_{22}=diag\{
	 m_{\mathcal I} k_{i} +\lambda_{i},\  i\in\mathcal O\}  .$$

We consider the second line in (\ref{tcf}); in this case it gives
\be
\label{uuu} \left( m_{\mathcal I} k_{i} +\li \right) u_i(0,t) = k_i\sum_{j\in \mathcal I} \mathcal W^N_j(t)\ ,\ \ \ i\in\mathcal O\ ;
\ee
then, summing the first $m_\mathcal I$ equations  of system (\ref{tcf})
we obtain
\be\label{sh}\sum_{j\in\mathcal O} k_j \sum_{i\in\mathcal I} \mathcal W^N_i(t) + m_\mathcal I\sum_{j\in\mathcal O}-k_j u_j(0,t)= \sum_{i\in\mathcal I} \li u_i(L_i,t)\ ;
\ee
using (\ref{uuu}) in (\ref{sh}) gives 
$$
\sum_{j\in\mathcal O} \frac {  \lambda_j k_j}{m_\mathcal I k_j+\lambda_j} \sum_{i\in\mathcal I} \mathcal W^N_i(t)
= \sum_{i\in\mathcal I} \li u_i(L_i,t)\ ,$$
then, by (\ref{uuu}) we obtain
$$
	\li u_i(0,t)= \frac{\li k_i}{m_\mathcal Ik_i+\li} \left(\sum_{j\in\mathcal O}\frac  {\lambda_j k_j}{m_{\mathcal I} k_{j} +\lambda_j}	\right)^{-1} \sum_{l\in\mathcal I} \lambda_l  u_l(L_l,t)\ ,\ \ \ i\in\mathcal O\ .
$$
It is easy 
to show  that, for every set   
$\{\gamma_i : i\in\mathcal O\ ,\ 
0<\gamma_i<1\ ,\ \dsp  \sum_{i\in\mathcal O}\gamma_i=1\}\  ,$
 there exist   sets
 $\{k_i: i\in\mathcal O\ ,\ k_i>0\}\ $
 satisfying 
\be\label{gammaalpha1}
 \dsp\gamma_i=\frac {\lambda_i k_i}{\lambda_i+m_\mathcal Ik_i}\left(\sum_{j\in\mathcal O}\frac  {\lambda_j k_j}{m_{\mathcal I} k_{j} +\lambda_j}\right)^{-1}\ ,\ \ \ i\in\mathcal O\ .\ee
For this, it is sufficient  to fix $\theta >0$ such that $m_{\mathcal I} \theta \gamma_i< \li$ for all $i\in\mathcal O$, and  choose each $k_i$ in such a way 
$$\frac {\lambda_ik_i} {\lambda_i+m_\mathcal I k_i}= \theta \gamma_i\ .$$

\bigskip 

  Notice that, when $\mathcal I=\{1,...,m-1\}\ ,$  $\ \mathcal O=\{m\}$, 
there is only the following  way to set  transmission conditions conserving the flux at the node, 
\be\label{n1}  \lambda_{m}u_m(0,t)=\sum_{i\in\mathcal I}  \li u_i(L_i,t)\ ;\ee
  for this reason, Theorem \ref{viscolim}
 ensures that, for this kind of  networks, for all the families  $\{K_{ij}\}$ satisfying (\ref{17}),  (\ref{anz})
 the sequences of solutions to problems (\ref{para2}) converge the solution to problem
 (\ref{ipe})-(\ref{bci2}), (\ref{n1}).
 
 \end{subsection}
\bigskip

 \begin{subsection}{}

 Another interesting case we can deal with, it is for  networks with  only two outgoing arcs, i.e.:
$$\mathcal I=\{1,2,...,m_I\}\ ,\ \mathcal O=\{h_1, h_2\}\ ,$$
with transmission conditions 
\be\label{nn} 
\lambda_{i} u_i(0,t)=\dsp\sum_{j\in\mathcal I} \gamma_{ij}\lambda_{j} u_j(L_j,t) 
\  \ \t{ for } \ i=h_1,h_2\
 \ ,\ee
where 
$ 
\gamma_{h_1j}\in (0,1)
$ and, obviously,   $
 \gamma_{h_2j}=1-  \gamma_{h_1j}$,  for  $j\in \mathcal I 
$.
Fixed a pair of  families $\{\gamma_{h_1j}\}_{j\in\mathcal I}$ and $\{\gamma_{h_2j}\}_{j\in\mathcal I}$
satisfying these  conditions, 
we are going to find families $\{K_{ij}\}$ achieving our purpose.
We impose the following conditions on some of coeffcients $K_{ij}$ involved in the transmission conditions
 in (\ref{para2}), 
$$K_{ij}=0 \ \  \t{ if }\ \ i\neq j ,\    i,j\in \mathcal I \t { or } i,j\in \mathcal O, \  ,\ \ \ \   K_{ih_2}=k >0\ \  \forall \   i\in \mathcal I ;$$ 
the corresponding 
coefficient matrix of the linear system (\ref{tcf}) is 
$$
Q=\begin{pmatrix}  K_{1h_1} +k& 0& .&.&0& -K_{1h_1} & -k\\
0&K_{2h_1} +k & .&.&0&  -K_{2h_1} & -k\\
	.&.&.&.&.&.&.\\
	.&.&.&.&.&.&.\\
	0&\dsp 0&.&. &   \dsp  K_{m_\mathcal Ih_1}+ k &  -K_{m_\mathcal Ih_1} &  -k
	\\ \\
	-K_{1h_1} &\dsp -K_{2h_1} &   
	 .&. &-K_{m_\mathcal Ih_1}&  \dsp \sum_{i\in\mathcal I}
	 K_{ih_1}  +\lambda_{h_1}   &  0 
	\\
	-k&\dsp -k & . &.& -k&0&  m_\mathcal Ik+\lambda_{h_2} 
	 \end{pmatrix}  .
$$

In the following we use the notations
$$u_i:=u_i(L_i,t) ,\  \mathcal W^N_i:= \mathcal W^N_i(t)
 \t{ for } i\in\mathcal I\ , \ \ \ u_{h_1}:=u_{h_1}(0,t), u_{h_2}:=u_{h_2}(0,t)\ 
 \ .$$
The first $m_\mathcal I$ equations of system  (\ref{tcf}) in this case give
\be
 (K_{ih_1}+k) \w_i = K_{ih_1} u_{h_1}+ku_{h_2} +\lambda_i u_i\ \ i\in\mathcal I \ ;
 \ee
  using the above relations in the penultimate  equation in system (\ref{tcf})  we obtain
 
 $$
 \left( \dsp  \sum_{i\in\mathcal I} \left(\frac{-K_{ih_1}^2}{ k+K_{ih_1}} + K_{ih_1} \right)
\frac 1 {\lambda_{h_1}}
+1\right) \lambda_{h_1}u_{h_1}
\qquad\qquad \qquad\qquad$$ $$\qquad \qquad
+\dsp  \sum_{i\in\mathcal I} \left(\frac{-K_{ih_1}}{ k+K_{ih_1}}\right) 
\frac  k{\lambda_{h_2}} \lambda_{h_2}u_{h_2} 
=\sum_{i\in\mathcal I} \frac{K_{ih_1}}{ k+K_{ih_1}}\lambda_iu_i\ ,
$$
 then, using the conservation of the flux (\ref{consflux}),  we have
$$
 \left(\dsp  \left(  \frac k {\lambda_{h_1}}+ \frac k {\lambda_{h_2}}\right) \sum_{i\in\mathcal I}\frac{K_{ih_1}}{ k+K_{ih_1}} 
  +1
  \right)
 \lambda_{h_1}u_{h_1} \qquad\qquad\qquad\qquad$$
$$\qquad\qquad\qquad \dsp = \sum_{i\in\mathcal I}    
\left(\frac  k{\lambda_{h_2}}
\sum_{j\in\mathcal I} \left(\frac{K_{jh_1}}{ k+K_{jh_1}}\right) 
+
 \frac{K_{ih_1}}{ k+K_{ih_1}}\right)
\lambda_i u_i 
$$ 
and the  transmission conditions (\ref{nn}) are verified if, for $i\in\mathcal I$,
$$
\gamma_{h_1 i}  \left(\dsp  \left(  \frac k {\lambda_{h_1}}+ \frac k {\lambda_{h_2}}\right) \sum_{i\in\mathcal I}\frac{K_{ih_1}}{ k+K_{ih_1}} 
  +1
  \right)-
  \left(\frac  k{\lambda_{h_2}}
\sum_{j\in\mathcal I} \left(\frac{K_{jh_1}}{ k+K_{jh_1}}\right) 
+
 \frac{K_{ih_1}}{ k+K_{ih_1}}\right)
=0 .
$$ 
 
 We set 
$$\theta(y)=\frac {y}{k+y}\ ,\ \ \ 
\theta_i=\theta(K_{ih_1})\ , \ \ \ i\in\mathcal I
\ ;$$ 
so,  for any set  $\{\gamma_{h_1i}\}_{i\in\mathcal I}$, we are   looking  for coefficients 
$K_{ih_1}>0$ and $k>0$ verifying 
\be\label{frigo2}
k \left(-\gamma_{h_1i} (\lambda_{h_1}+\lambda_{h_2})
+\lambda_{h_1} \right) 
\dsp \sum_{j\in\mathcal I}
\theta_j 
+ \theta_i
 \lambda_{h_1}\lambda_{h_2}
=  \gamma_{h_1i}\lambda_{h_1}\lambda_{h_2}  \ , \ \ i\in\mathcal I\ .
\ee

We consider the linear system for the unknowns $X_i$
\be\label{legale2}
 k \left(-\gamma_{h_1i} (\lambda_{h_1}+\lambda_{h_2})
+\lambda_{h_1} \right)
 \dsp \sum_{j\in\mathcal I}
X_j   + X_i
 \lambda_{h_1}\lambda_{h_2}
=  \gamma_{h_1i}\lambda_{h_1}\lambda_{h_2}  \ , \ \ i\in\mathcal I\ ;
\ee
for small $k$ the system has dominant diagonal, so it has  a unique solution $\{X_i\}_{i\in\mathcal I}$. 
Notice that for $y> 0$ the function $ \theta$ increcreases  and $ \theta(y)\in(0,1)$,
then  we are going to show that $0<X_i<1$ for all $i\in\mathcal I$.

 We sum the equations in (\ref{legale2}) 
$$ 
 \dsp \sum_{j\in\mathcal I}
X_j  \left(  \dsp  k \sum_{i\in\mathcal I} (-\gamma_{h_1i} (\lambda_{h_1}+\lambda_{h_2})
+\lambda_{h_1}) + \lambda_{h_1}\lambda_{h_2} \right)=  
\lambda_{h_1}\lambda_{h_2}
\sum_{i\in\mathcal I} \gamma_{h_1i}\ 
$$

and we use the above equality in (\ref{legale2})
$$
X_i
 =  \gamma_{h_1i} - k \  \frac{  \left(-\gamma_{h_1i} (\lambda_{h_1}+\lambda_{h_2})
+\lambda_{h_1} \right)  
\dsp\sum_{i\in\mathcal I} \gamma_{h_1i}}
{  \dsp  \lambda_{h_1}\lambda_{h_2}+ k \sum_{j\in\mathcal I} (-\gamma_{h_1j} (\lambda_{h_1}+
\lambda_{h_4})
+\lambda_{h_1}) }   
  \ , \ \ i\in\mathcal I\ ;
$$
now, since each $0<\gamma_{h_1i}<1$,  it is possible to choose $k$ so small to have 
$0<X_i<1$ for all $i\in\mathcal I$.

 It follows that, to each set  $\{\gamma_{h_1i}\}_{i\in\mathcal I}$, $ \gamma_{h_1i}\in(0,1)$,
  corresponds a small value  $k_0$, such that, for each 
 $0<k<k_0$ there exist $K_{ih_1}$, $i\in\mathcal I$,  verifying (\ref{frigo2}).
  \end{subsection}
 
\end{section}

\begin{section}{Appendix}

 { \it Proof of Lemma \ref{appr}. }
Let $v\in BV(\A)$. We consider a sequence $\{ w_n\}_{n\in\N}$ such that 
$ {w_n}_i \in C^1(\ov I_i) \cap W^{2,1}(I_i)$ and 
\be \label{w} \begin{array}{ll}
  \dsp   \Vert  {w_n} -v \Vert_{L^1(\mathcal A)} \to_{n\to \infty} 0\ ,\\  \Vert { {w'_n}_i} \Vert_{L^1(I_i)}\leq TV_0^{L_i} (v_i)\ \  \forall i\in\M\ , \ \qquad
     \ep_n \Vert w_n\Vert_{W^{2,1}(\A)} \leq  C_1\ ,\end{array}
\ee
 where $C_1$ is a quantity independent from $n$ and , for $f\in BV(\mathcal A) $, 
$$TV_0^{L_i}(f_i)=\sup \left\{\int_{I_i} f_i \phi' dx : \phi\in C^1_0(I_i), |\phi|\leq 1\right\} $$
\cite{EG}.
Then we introduce the polinomials $p_{n}$ 
defined on  the network,
\be\label{poli} p_{n_i}(x) = a^i_{n} x^3 +b^i_{n} x^2 + c^i_{n}x +d_n^i \qquad x\in I_i\ ,\ee
whose coefficients have to be  determined.
When $i\in \mathcal O$,  we impose the  conditions 
\be
       \label{pn}\left\{ \begin{array}{ll}
              {p_n}_i(0)= {w_n}_i (0)\ , \\ 
        \ep_n {p'_n}_i(0) =  \dsp  \sum_{j\in\M} \alpha_{ij} w_{n_j}(N) + \li  {w_n}_i(0)  
             \\ 
         {p_n}_i(\delta_n) = {w_n}_i (\delta_n)\ , \\ 
         {p'_n}_i(\delta_n)={w'_n}_i (\delta_n)\ 
         \end{array}\right.
  \ee
 where 
 \be
     \label{tetan}
     \delta_n= 
         \ep_n^{\theta}\ ,
         \ \qquad \theta>1 \ ,
    \ee  
    and we define the sequence  $\{v_{n}\}_{n\in\N}$ on the outgoing arcs 
 $$
     {v_n}_i(x)=\left\{ \begin{array} {ll}{w_n}_i (x)\qquad x\in [\delta_n, L_i] \\ 
      {p_n}_i(x) \qquad x\in [0, \delta_n]\end{array} \qquad i\in \mathcal O\ . \right.
 $$
 When $ i\in  \mathcal I $ we define
 $$
      {v_n}_i(x)=\left\{ \begin{array} {ll}{w_n}_i (x)\qquad x\in [\ep_n, L_i -\delta_n] \\ 
      {p_n}_i(x) \qquad x\in [L_i-\delta_n,L_i] \\ 
      {r_n}_i(x) \qquad x\in [0,\ep_n] 
     \end{array} \qquad i\in \mathcal I\ , \right.
 $$
where ${p_n}_i(x) $ are the polinomials in (\ref{poli}) whose coefficients are determined by 
 \be
      \label{pn2}\left\{ \begin{array}{ll}
              {p_n}_i(L_i)= {w_n}_i(L_i)\ , \\ 
       \ep _n{p'_n}_i(L_i) =  \dsp - \sum_{j\in\M} \alpha_{ij} {w_n}_j(N) + \li  {w_n}_i(L_i) \\ 
      {p_n}_i(L_i-\delta_n) = {w_n}_i (L_i-\delta_n)\ , \\ 
      {p'_n}_i(L_i-\delta_n)={w'_n}_i(L_i-\delta_n)\  ,
 \end{array}\right.\ee
 and ${r_n}_i(x)=\mu^i_n x^2+\nu^i_n x+\rho_n^i$ , where the coefficients are determined by the 
 following conditions 
 \be
       \label{rn}\left\{ \begin{array}{ll}
               {r_n}_i(0)= B_i\ , \\ 
                {r_n}_i(\ep_n) = {w_n}_i (\ep_n)\ , \\ 
         {r'_n}_i(\ep_n)={w'_n}_i (\ep_n)\  .
         \end{array}\right.
  \ee

 For all $n\in \N$ , $v_n\in D_{\ep_n}$, since ${v_n}_i\in \mathcal C^1(I_i)$ for all $i\in\M$ and  $v''_n\in L^1(\A)$, the transmission conditions at the internal node are verified, thanks to the first two equalities in (\ref{pn}) and (\ref{pn2}) and also the boundary ones, thanks to
  the first condition in (\ref{rn}).

 If $i\in\mathcal O$, the conditions in (\ref{pn}) imply
\be
          \dsp  d_n^i= w_{n_i}(0) 
         \ ,\qquad
        \ep_n c_n^i = \dsp  \sum_{j\in\M} \alpha_{ij} w_{n_j}(N) + \li  w_{ni}(0) 
       \ee
so that 
\be
    \label{cd} |d_n^i|\ ,  \ep_n |c_n^i| \leq C_0\ ,\qquad n\in\N \ ,
\ee 
where $C_0$ is a quantity  independent from $n$, thanks to (\ref{w}).
The conditions in (\ref{pn}) also imply  
  \be
       \label{b}
       \vert b_n^i\vert  \delta_n^{2}= \vert -2c_n^i\delta_n + 3 ({w_n}_i(\delta_n)-{w_n}_i(0)) - \delta_n {w'_n}_i(\delta_n)\vert 
         \ee
  \be
       \label{a} 
        \vert a_n^i  \vert \delta_n^{3}  =  \vert  {c_n^i} \delta_n- 2  ({w_n}_i(\delta_n)-{w_n}_i(0)) +\delta_n {w'_n}_i(\delta_n) \vert \ ;
            \ee
    notice that the quantities in (\ref{b}) and (\ref{a}) go to zero when $n$ goes to infinity,
        since
    ${w_n}_i$ is continuous and $\delta_n {w'_n}_i(\delta_n)$ is infinitesimal thanks (\ref{tetan}) and (\ref{w}).
        
    Now we have 
  $$
      \Vert p_{n_i}\Vert_{L^{1}(0,\delta_n)} \leq
      \frac{ | a_n^i|  \delta_n^{4}}{4} + \frac{|b_n^i| \delta_n^{3}} 3 +\frac{ |c_n^i| \delta_n^{2}}{2} +|d_n^i| \delta_n
       $$
  $$
      \Vert p'_{n_i}\Vert_{L^{1}(0,\delta_n)} \leq | a_n^i|  \delta_n^{3}+ |b_n^i| \delta_n^{2} +\vert c_n^i\vert  \delta_n \ .
     $$
Similar computations can be made when $i\in\mathcal I$,
 so that,
 thanks to 
  (\ref{cd}),(\ref{b}), (\ref{a}), 
\be
	\label{p1}      \sum_{i\in\mathcal O}  \Vert p_{n_i} \Vert_{W^{1,1}(0,\delta_n)}  +
          \sum_{i\in\mathcal I}  \Vert p_{n_i} \Vert_{W^{1,1}(L_i-\delta_n,L_i)} 
                \dsp  \rightarrow_{n\to +\infty} 0\ .
 \ee
As regard to $r_n$, using conditions in (\ref{rn})
 we see that
$$
	\rho_n^i=B_i\ ,
$$
$$
	\mu_n^i \ep_n^2 + \nu_n^i \ep_n + B_i={w_n}_i (\ep_n)\ ,
$$
$$
	2\mu_n^i \ep_n +\nu_n^i={w'_n}_i (\ep_n)\ ,
$$
so we have, thanks to (\ref{w}), 
$$
	|\nu_n^i \ep_n |= |-2 B_i +2{w_n}_i (\ep_n)-{w'_n}_i (\ep_n)\ep_n |\leq C_\nu\ ,
	\quad  C_\nu \t{ independent from } n\ ,
$$
$$
	|\mu_n^i \ep_n ^2|= |B_i+ {w'_n}_i (\ep_n)\ep_n-{w_n}_i (\ep_n)|
	\leq C_\mu\ ,
	\quad  C_\mu \t{ independent from } n\ ,
$$
which imply
\be
	\label{ri}
	\sum_{i\in\mathcal I}\Vert r_{n_i} \Vert_{L^{1}(0,\ep_n)} \leq 
	 | \mu_n^i |\frac  {\ep_n^3} 3  +|\nu_n^i| \frac{\ep_n^2}2  +|\rho_n^i| \ep_n
                  \to_{n\to +\infty} 0\ ,
\ee
and
\be
	\label{r'1}
	\sum_{i\in\mathcal I}\Vert r'_{n_i} \Vert_{L^{1}(0,\ep_n)} \leq 
	 | \mu_n^i | \ep_n^2 +|\nu_n^i| \ep_n 
          \leq C_r\ ,
	\quad  C_r \t{ independent from } n\ .
\ee

Thanks to(\ref{w}),(\ref{p1}),(\ref{ri}) we obtain 
 \be
        \label{L1}   \Vert v_{n} -v\Vert_{L^{1}(\A)}  \dsp \rightarrow_{n\rightarrow +\infty} 0\ .
         \ee
 
 Now we are going to prove that, for $i\in\M$,
 \be
 	\label{bv}
 	 \Vert{{v'_n}_i}\Vert_{L^1(I_i)} \leq TV_0^{L_i} (v_i) 
	+  C_{BV}, \  \quad
 C_{BV} \t{  independent from }n\ :
 \ee
 for $i\in\mathcal I$ 
$$
	\int_{I_i} |{{v'_n}_i }| dx \leq 
	\Vert  {p'_n}_i \Vert_{L^{1}(L_I-\delta_n,L_i)} 
	+  \Vert  {r'_n}_i \Vert_{L^{1}(0,\ep_n)} +\Vert {{w'_n}_i}  \Vert_{L^{1}(I_i)}\ ,
$$
and similar computations can be made for the outgoing arcs, 	so that, 
using (\ref{w}) and 
(\ref{p1})-(\ref{r'1})
 we obtain (\ref{bv}), which implies  that 
 $$\Vert {v_n}\Vert_{W^{1,1}(\A)} \leq C_2\ ,\qquad    \t{ for all }n\in\N\ ,$$
  where $C_2$  depends on
 $\Vert {v}\Vert_{L^1(\A)}$ and $TV_0^{L_i}({v}_i)$, $i\in\M$.

We have only to prove that $\ep_n \Vert v_n \Vert _{W^{2,1}(\A)} $ is bounded independently of $ n$.
  For $i\in\mathcal O$ ,
$$
   \ep_n \Vert {p''_n}_i \Vert_{L^1(0,\delta_n)} \leq \ep_n (3 |a_n^i|  \delta_n^2 +2 |b_n^i| \delta_n);
$$
from   (\ref{b}), (\ref{a}),  (\ref{cd}) and  (\ref{w}) we know that
$$
       \ep_n |b_n^i| \delta_n =\ep_n |-2c_n^i + \frac 3 {\delta_n} ({w_n}_i(\delta_n)-{w_n}_i(0)) - {w'_n}_i(\delta_n)|
$$
$$
     \leq  2  C_0+ \ep_n 4 C_S \Vert  {w_n}_i \Vert _{W^{2,1}(I_i)} 
      \leq  2C_0 +4C_SC_1
$$
    and 
$$
       \ep_n  | a_n^i|  \delta_n^{2} = \ep_n \left\vert   {c_n^i} - 
       2\frac{{w_n}_i(\delta_n)-{w_n}_i(0)}{\delta_n}+ {w'_n}_i(\delta_n)\right\vert\ 
$$
$$
        \leq C_0+ 3 \ep_n C_S \Vert  {w_n}_i \Vert _{W^{2,1}(I_i)} \leq  C_0 +3 C_SC_1\ ,
$$
where $C_S$ depends on Sobolev constants. Similar estimates can be obtained for  $i\in\mathcal I$, 
to conclude that there exists $C_3>0$ such that 
$$
   	\ep_n \Vert  {p_n}_i\Vert_{W^{2,1}(0,\delta_n)}\  ,\   \ep_n \Vert  {p_n}_j\Vert_{W^{2,1}(L_j-\ep_n, L_j)}\leq  C_3\ ,\qquad \  \forall n\in\N\ , 	i\in\mathcal O\ ,
  	 \ j\in\mathcal I,
$$
(see also (\ref{p1}); moreover, for $i\in\mathcal I$, 
we  know  that 
$$
	\ep_n \Vert {r_n}_i''\Vert_{L^1(0,\ep_n)}\leq 2\ep^2_n| \mu_n^i | \leq
	2 C_\mu\ .
$$
Since 
  $\dsp \ep_n \Vert  w_n\Vert_{W^{2,1}(\A)} \leq  C_1$, for all  $n\in\N$, 
  we conclude that 
  $\dsp \ep_n \Vert v_n\Vert_{W^{2,1}(\A)} \leq C$ , for all $n\in\N\ $, where $C$ depends on 
  $C_0, C_1, C_2, C_3, C_\mu$.

\end{section}

\end{document}